\documentclass[11pt, oneside]{article}   	% use "amsart" instead of "article" for AMSLaTeX format
\usepackage{geometry}                		% See geometry.pdf to learn the layout options. There are lots.
\usepackage{graphicx}				% Use pdf, png, jpg, or eps§ with pdflatex; use eps in DVI mode
\usepackage{amsmath,amssymb,amsthm}
\usepackage{contour}
    \contourlength{0.2em}
\usepackage{pgfplots}
\usepackage{hyperref}
\usepgfplotslibrary{fillbetween}

\usepackage{hyperref}

\newcommand{\Dim}{\mathrm{Dim}}
\newcommand{\uh}{\upharpoonright}

\newcommand{\fs}{\{0,1\}^*}

\newcommand{\EXT}{\mathsf{EXT}}
\newcommand{\EXTp}{\mathsf{EXT^{p}}}

\newcommand{\N}{\mathbb{N}}

\newcommand{\dom}{\mathrm{dom}}

\newtheorem{theorem}{Theorem}[section]
\newtheorem{lemma}[theorem]{Lemma}
\newtheorem{fact}[theorem]{Fact}
\newtheorem{corollary}[theorem]{Corollary}
\newtheorem{proposition}[theorem]{Proposition}
\newtheorem{remark}[theorem]{Remark}

\theoremstyle{definition}
\newtheorem{definition}[theorem]{Definition}

\usepackage[backgroundcolor=orange!50, textsize=small]{todonotes}

\title{Optimal bounds for single-source Kolmogorov extractors}
\author{Laurent Bienvenu\thanks{Supported by ANR-15-CE40-0016-01 RaCAF grant}, Barbara F. Csima\thanks{Partially supported by Canadian NSERC Discovery Grant 312501.}, Matthew Harrison-Trainor\thanks{Supported by a Canadian NSERC Banting fellowship.}}
\date{}							% Activate to display a given date or no date

\begin{document}
\maketitle
%\section{}
%\subsection{}

\begin{abstract}
The rate of randomness (or dimension) of a string $\sigma$ is the ratio $C(\sigma)/|\sigma|$ where $C(\sigma)$ is the Kolmogorov complexity of $\sigma$. While it is known that a single computable transformation cannot increase the rate of randomness of all sequences, Fortnow, Hitchcock, Pavan, Vinodchandran, and Wang showed that for any $0<\alpha<\beta<1$, there are a finite number of computable transformations such that any string of rate at least $\alpha$ is turned into a string of rate at least $\beta$ by one of these transformations. However, their proof only gives very loose bounds on the correspondence between the number of transformations and the increase of rate of randomness one can achieve. By translating this problem to combinatorics on (hyper)graphs, we provide a tight bound, namely: Using $k$ transformations, one can get an increase from rate $\alpha$ to any rate $\beta < k\alpha/(1+(k-1)\alpha)$, and this is optimal.
\end{abstract}

\section{Introduction}

For a finite binary string $\sigma$, the (plain) Kolmogorov complexity $C(\sigma)$ is the length of the shortest program, written in binary and for a fixed universal interpreter, which outputs $\sigma$. The quantity $C(\sigma)$ can range from $0$ to $|\sigma|+d$ for a fixed constant~$d$, and the closer it is to $|\sigma|$, the more random the string $\sigma$ will look (in the sense that it will look like the typical sequence of random bits where bits are chosen independently and with probability $1/2$ to be equal to $0$).

One can normalize by the length of $\sigma$ and consider the quantity $C(\sigma)/|\sigma|$, which measures the \emph{rate of randomness}. This corresponds fairly well to our intuition of partial randomness: consider for example a binary string of length~$3n$ where every third bit is chosen at random and then doubled, like $000111000000111000\ldots$. One would expect this sequence to have a rate of randomness of $\approx 1/3$, and this is indeed what will happen with high probability.

This idea can be extended to infinite binary sequences~$X$, by considering the asymptotic behaviour of $C(X \uh n)/n$, where $X \uh n$ is the $n$-bit prefix of~$n$. As this quantity may not converge, it is natural to consider both
\[
\dim(X) = \liminf_{n \rightarrow \infty}  \frac{C(X \uh n)}{n}
\]
and
\[
\Dim(X) = \limsup_{n \rightarrow \infty}  \frac{C(X \uh n)}{n}
\]
respectively called effective Hausdorff dimension and effective packing dimension of~$X$ (the reason for these names are the close connections between randomness rates and fractal dimensions, see for example~\cite[Chapter 13]{DowneyHirschfeldt} for an extensive presentation of the topic; by extension, for a finite string $\sigma$, the rate of randomness $C(\sigma)/|\sigma|$ is sometimes referred to as the \emph{dimension of $\sigma$}).

Since one can think of a sequence of dimension strictly between $0$ and $1$ to be partially but imperfectly random, one natural question is whether one can `extract randomness' from it. More specifically, can every such sequence $X$ Turing-compute a sequence~$Y$ of dimension~$1$, or close to~$1$, or at least of dimension greater than that of~$X$? This natural question was first formulated in 2004 by Reimann~\cite{Reimann2004} and sparked an intense line of research in the following years. It turns out that the answer depends on which of the two above notions of dimension one considers. For effective Hausdorff dimension, a full negative answer was given by Miller~\cite{Miller2011}.

\begin{theorem}[Miller]
For any rational $q \in [0,1]$, there exists an infinite binary sequence~$X$ such that $\dim(X)=q$ and any infinite binary sequence~$Y$ Turing-computed by~$X$ has $\dim(Y) \leq q$.
\end{theorem}

On the other hand, effective packing dimension is amenable to extraction. Indeed, using deep results from pseudo-randomness in computational complexity~\cite{BarakIW2006}, Fortnow et al.\ proved the following.

\begin{theorem}[Fortnow et al.~\cite{FortnowHPVW2006}]\label{thm:fortnow-et-al}
If $\Dim(X) > 0$, for any $\varepsilon >0$, $X$ computes a~$Y$ such that $\Dim(Y)>1-\varepsilon$. Moreover, the reduction from $X$ to $Y$ is an exponential-time reduction, hence a tt-reduction.
\end{theorem}

(Bienvenu et al.~\cite{BienvenuDotyStephan} independently obtained the first part of the theorem with a more direct proof, but with a reduction from $X$ to $Y$ that is not even guaranteed to be wtt). Conidis~\cite{Conidis12} showed that Fortnow et al.'s theorem cannot be strengthened to $\Dim(Y)=1$, even for Turing reductions.

As an intermediate step towards the proof of Theorem~\ref{thm:fortnow-et-al}, which concerns infinite binary sequences, Fortnow et al. obtained a result of independent interest in the case of finite strings.

\begin{theorem}[Fortnow et al.~\cite{FortnowHPVW2006}]\label{thm:fortnow-et-al-2}
Let $0< \alpha <\beta <1$. There exists a polynomial-time function $E(.,.)$, a linear function $f$ and a constant $h$ such that for every, $n$, for every $\sigma$ of length $f(n)$ such that $C(\sigma) \geq \alpha |\sigma|$, there exists a string $a_\sigma$ of length~$h$ such that $\tau=E(\sigma,a_\sigma)$ has length~$n$ and $C(\tau) \geq \beta |\tau|$.
\end{theorem}

\bigskip

This is interesting because for any $\alpha < \beta$, there is no computable function $F$ with only one argument and computable function~$f$ such that for every $\sigma$ of length $f(n)$ such that $C(\sigma) \geq \alpha |\sigma|$, $\tau=F(\sigma)$ has length~$n$ and $C(\tau) \geq \beta |\tau|$. (This result seems to be well-known but a full proof is hard to find in the literature. In any case it follows from our results). Therefore, just a few extra bits of extra information (or `advice') makes all the difference if we wish to effectively increase the rate of randomness of individual strings.

More generally, a procedure whose goal is to turn a string or tuple of strings of a given rate of randomness to a string of higher rate of randomness is called a \emph{Kolmogorov extractor}, a term coined by Zimand, who made important contributions to the study of this concept, in particular Kolmogorov extractors with two sources (i.e., two input strings~$x$ and $y$); see the survey~\cite{Zimand2010}. Zimand also studied in~\cite{Zimand2011} single-source Kolmogorov extractors (like the function $E$ of Theorem~\ref{thm:fortnow-et-al-2}), for which the most natural question is how the amount of advice relates to the increase of rate of randomness one can obtain. He showed in particular that earlier results of Vereshchagin and Vyugin~\cite{VereshchaginVyugin} already give an upper bound:

\begin{theorem}[Zimand~\cite{Zimand2011}, based on~\cite{VereshchaginVyugin}]\label{thm:vv}
Let $0<\alpha < \beta <1$ and suppose there is a partial computable function $E(.,.)$, a linear function $f$, and a constant $h$ with the property that for every, $n$, for every $\sigma$ of length $f(n)$ such that $C(\sigma) \geq \alpha |\sigma|$, there exists a string $a_\sigma$ of length~$h$ such that $\tau=E(\sigma,a_\sigma)$ has length~$n$ and $C(\tau) \geq \beta |\tau|$. Then \[\beta \leq 1 - \frac{1-\alpha}{2^{h+1}-1} + o(1).\] %$\beta \leq 1 - (1-\alpha)/2^h$.
\end{theorem}

The goal of this paper is to refine this theorem and get an exact correspondence between the amount of advice $h$ and the rate increase $\alpha \rightarrow \beta$ one can get. We note that allowing an advice of size $h$ is like having a family of $2^h$ partial computable functions $\{E(.,a) \mid |a|=h\}$. In order to have a more fine-grained analysis, we consider the case where we have $k$ functions, where $k$ is not necessarily a power of~$2$. We begin by assuming that each of the $k$ functions is total, which corresponds to asking that $E(.,a)$ converges for every $a$; later we will allow the $k$ functions to be partial, which is exactly equivalent to computation with small advice. We thus propose the following definition.

\begin{definition}
For $k\geq 1$, let $\EXT(k)$ be the set of pairs of reals $(\alpha, \beta)$ such that $\alpha, \beta \in [0,1]$ and for which there exist a total one-to-one computable function $f: \N \rightarrow \N$, $k$ total computable functions $\Gamma_1, \ldots, \Gamma_k: \fs \rightarrow \fs$, and a constant $d \in \N$ with the following property: For all~$n$, and every string $\sigma$, if $|\sigma|=f(n)$, then $|\Gamma_i(\sigma)|=n$ for all $i \leq k$, and if furthermore $C(\sigma) \geq \alpha|\sigma|+d$, then for some $i$, $C(\Gamma_i(\sigma)) \geq \beta |\Gamma_i(\sigma)| - d$.
\end{definition}

(Kolmogorov complexity being defined up to an additive constant, which depends on the choice of universal machine, the use of the constant~$d$ in our definition ensures that $\EXT(k)$ does not depend on the particular choice of universal machine).

\noindent Essentially, $(\alpha,\beta) \in \EXT(k)$ if, for each $n$, one can computably transform each string $\sigma$ of length $f(n)$ into $k$ strings $\tau_1,\ldots,\tau_k$ of length $n$ such that if $\sigma$ had dimension at least $\alpha$, then at least one of the $\tau_i$ has dimension at least $\beta$. That is, one can extract dimension $\beta$ from strings of dimension $\alpha$ using $k$ functions.

An easy argument using information conservation gives us a lower bound for $f(n)$.

\begin{remark}\label{rem:lower-bound-for-f}
If $d$, $f$, and $(\Gamma_i)$ witness that $(\alpha,\beta) \in \EXT(k)$, then the function~$f$ must be such that $f(n) \geq (\beta/\alpha) n - O(1)$ for all~$n$. Indeed, for a given~$n$, take a $\sigma$ such that $|\sigma|=f(n)$ and $C(\sigma)=\alpha f(n)+O(1)$ (there is always such a $\sigma$). On the one hand we have $C(\Gamma_i(\sigma)) \geq \beta n - O(1)$ for some~$i$ by the assumption on the $\Gamma_i$. On the other hand, by information conservation, $C(\Gamma_i(\sigma)) \leq C(\sigma) + O(1) \leq \alpha f(n) + O(1)$. Putting the two together gives us $f(n) \geq (\beta/\alpha)n - O(1)$.
\end{remark}

As announced above, we will obtain a precise characterization of $\EXT(k)$, namely we will prove the following.

\begin{theorem}\label{thm:main-total}
$(\alpha, \beta) \in \EXT(k)$ if and only if one of the following holds:
\begin{itemize}
	\item $k = 1$ and $\beta \leq \alpha$, or
	\item $k \geq 2$ and either $\alpha = \beta = 0$, $\alpha = \beta = 1$, or
		\[ \beta <  \frac{k\alpha}{1+(k-1)\alpha}.\]
\end{itemize}
\end{theorem}

Note that $(\alpha,\beta) \in \EXT(k)$ when $\alpha = \beta$ is trivial: it suffices to take $f(n) = n$, $d = 0$, and the identity function $\Gamma(\sigma) = \sigma$. Note also that when $k = 1$, then the expression $\frac{k\alpha}{1+(k-1)\alpha}$ is equal to $\alpha$; so the difference between the cases $k = 1$ and $k \geq 2$ is that in the former we get $\beta \leq \frac{k\alpha}{1+(k-1)\alpha}$ and in the latter we get $\beta < \frac{k\alpha}{1+(k-1)\alpha}$. The following plot shows the graph of $\beta = \frac{k\alpha}{1+(k-1)\alpha}$ for $k = 2$.

\begin{center}
\begin{tikzpicture}
\begin{axis}[
		ylabel style={rotate=-90},
    axis lines = left,
    xlabel = {$\alpha$},
    ylabel = {$\beta$},
		title = {$\beta < \frac{k\alpha}{1+(k-1)\alpha}$ for $k = 2$},
]

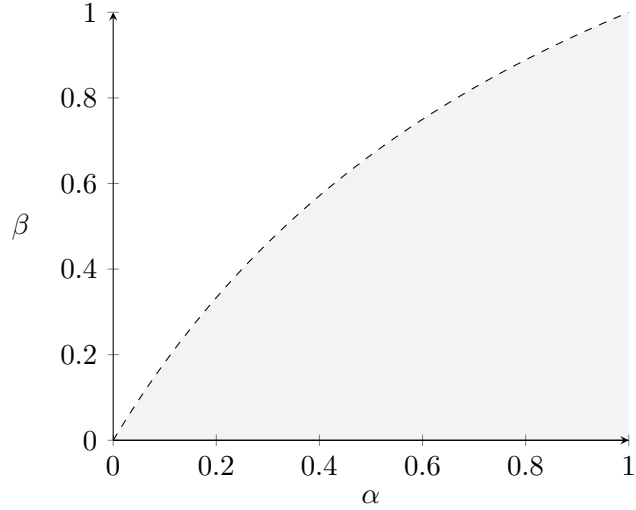
\addplot [
		dashed,
		name path = k2,
    domain=-0:1,
    samples=100,
    color=black,
]
{2*x / (1 + x)};

\addplot [
		name path = axis,
    domain=-0:1,
    samples=100,
    color=black,
]
{0};

    \addplot [
        thick,
        color=black,
        fill=black,
        fill opacity=0.05
    ]
    fill between[
        of=k2 and axis,
        soft clip={domain=0:1},
    ];
\end{axis}
\end{tikzpicture}
\end{center}

As $k$ gets larger, we can extract more and more dimension. In line with Theorem~\ref{thm:fortnow-et-al} above, as $k \to \infty$, $\frac{k\alpha}{1+(k-1)\alpha} \to 1$, and so with a large number of extractors one can extract almost-random strings.

\begin{center}
\begin{tikzpicture}
\begin{axis}[
		ylabel style={rotate=-90},
    axis lines = left,
    xlabel = $\alpha$,
    ylabel = {$\beta$},
		title = {$\beta = \frac{k\alpha}{1+(k-1)\alpha}$ for various values of $k$}
]
%Below the red parabola is defined

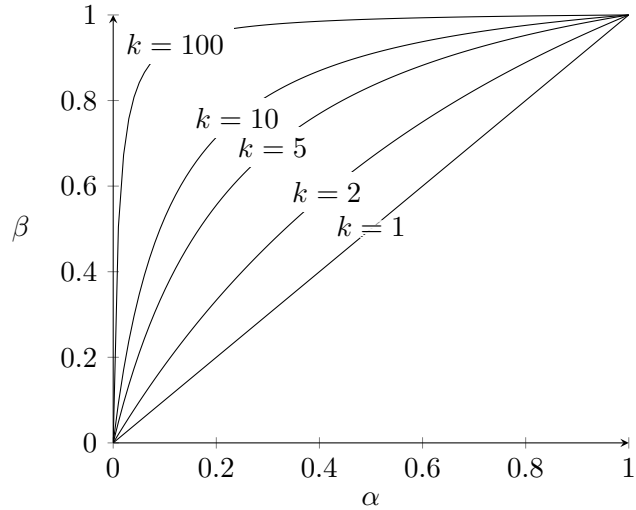
\addplot [
		name path = k1,
    domain=-0:1,
    samples=100,
    color=black,
]
{x}
node[pos=0.5] {\contour{white}{$k=1$}};

\addplot [
		name path = k2,
    domain=-0:1,
    samples=100,
    color=black,
]
{2*x / (1 + x)}
node[pos=0.5, fill=white] {$k=2$};

%\addplot [
		%name path = k3,
    %domain=-0:1,
    %samples=100,
    %color=black,
%]
%{3*x / (1 + 2*x)}
%node[right,pos=0.5] {$k=3$};

\addplot [
		name path = k5,
    domain=-0:1,
    samples=100,
    color=black,
]
{5*x / (1 + 4*x)}
node[pos=0.5, fill=white] {$k=5$};

\addplot [
		name path = k10,
    domain=-0:1,
    samples=100,
    color=black,
]
{10*x / (1 + 9*x)}
node[pos=0.5, fill=white] {$k=10$};

%\addplot [
		%name path = k50,
    %domain=-0:1,
    %samples=100,
    %color=black,
%]
%{50*x / (1 + 49*x)}
%node[right,pos=0.5] {$k=50$};

\addplot [
		name path = k100,
    domain=-0:1,
    samples=100,
    color=black,
]
{100*x / (1 + 99*x)}
node[pos=0.52, fill=white] {$k=100$};

\end{axis}
\end{tikzpicture}
\end{center}

In Remark \ref{rem:lower-bound-for-f} above we said that the function $f$ which witnesses that $(\alpha,\beta) \in \EXT(k)$ must satisfy $f(n) \geq (\beta/\alpha) n - O(1)$. In fact, this is optimal; one can witness that $(\alpha,\beta) \in \EXT(k)$ using a function $f(n) = (\beta/\alpha) n - O(1)$. Moreover, from Proposition \ref{prop:controlling-f}, the following inequality holds:
\[ (\beta/\alpha) n - O(1) \leq f(n) \leq \frac{1-\beta}{1-\alpha}kn +O(1). \]
As $\beta \to \frac{k\alpha}{1+(k-1)\alpha}$, we have that $\frac{1-\beta}{1-\alpha}k \to \frac{\beta}{\alpha}$, and so in some sense $f(n) = (\beta/\alpha) n - O(1)$ is optimal.

We do not know if the functions $\Gamma_1,\ldots,\Gamma_k$ which witness that $(\alpha,\beta) \in \EXT(k)$ can be polynomial-time. In Theorem \ref{thm:fortnow-et-al-2}, the extractors were polynomial time, so we know that one can extract at least some dimension with polynomial-time extractors, but what we do not know is whether polynomial time extractors can be optimal. This may be a difficult question, as our construction of optimal extractors passes through a probabilistic construction of hypergraphs. Indeed, Theorem \ref{thm:graph-correspondence} shows that $(\alpha,\beta) \in \EXT(k)$ is equivalent to the existence of a sequence of $k$-hypergraphs whose edges are well spread out in a particular sense which is related to (but not the same as) the jumbled graphs introduced by Thomason \cite{Thomason87a,Thomason87b}. So whether one can find polynomial-time extractors which are optimal is equivalent to finding an efficient construction of these hypergraphs. There are long-standing open problems which ask similar questions. For example, one such open problem is finding an efficient construction of a graph of size $n$ with no cliques or independent sets of size $c \log n$. Such graphs give bounds on the Ramsey numbers and their existence can be proved using the probabilistic methods. See  \cite{Chung}.

The dimension extractors which we have been considering have all been total functions. One could potentially improve the extractors by allowing them to be partial. One way that this might help is that, say with $k = 2$, on input $\sigma$, $\Gamma_1$ could search for a short description of $\sigma$ and then compute an output based on that, while $\Gamma_2$ could assume that $\sigma$ has no short description and so has relatively high Kolmogorov complexity. In this case, $\Gamma_1$ would be undefined if $\sigma$ has no short description. Thus we define $\EXTp(k)$, the set of pairs $(\alpha,\beta)$ such that we can extract dimension $\beta$ from strings of length $\alpha$ using $k$ \textit{partial} functions.

\begin{definition}
For $k\geq 1$, let $\EXTp(k)$ be the set of pairs of reals $(\alpha, \beta)$ such that $\alpha, \beta \in [0,1]$ and for which there exist a total one-to-one computable function $f: \N \rightarrow \N$, $k$ partial computable functions $\Gamma_1, \ldots, \Gamma_k: \fs \rightarrow \fs$, and a constant $d \in \N$ with the following property: For all~$n$, and every string $\sigma$, if $|\sigma|=f(n)$, then $|\Gamma_i(\sigma)|=n$ for all $i \leq k$ for which $\Gamma_i(\sigma)$ is defined, and if furthermore $C(\sigma) \geq \alpha|\sigma|+d$, then for some $i$, $\Gamma_i(\sigma)$ is defined and $C(\Gamma_i(\sigma)) \geq \beta |\Gamma_i(\sigma)| - d$.
\end{definition}

With the same argument as before, we get a lower bound on the function $f$ which can witness that $(\alpha,\beta) \in \EXTp(k)$.

\begin{remark}\label{rem:lb2}
If $d$, $f$, and $(\Gamma_i)$ witness that $(\alpha,\beta) \in \EXTp(k)$, then the function~$f$ must be such that $f(n) \geq (\beta/\alpha) n - O(1)$ for all~$n$.
\end{remark}

We also get a precise characterization of $\EXTp(k)$ wherein it turns out that using partial function gets us only a very slight improvement.

\begin{theorem}\label{thm:main-partial}
$(\alpha, \beta) \in \EXTp(k)$ if and only if one of the following holds:
\begin{itemize}
	\item $k = 1$ and $\alpha \leq \beta$,
	\item $k \geq 2$ and $\beta <  \frac{k\alpha}{1+(k-1)\alpha}$, or
	\item $k \geq 2$, $\beta =  \frac{k\alpha}{1+(k-1)\alpha}$, and $\alpha$ and $\beta$ are computable.
\end{itemize}
\end{theorem}

\noindent If $\beta =  \frac{k\alpha}{1+(k-1)\alpha}$, then $\alpha = \frac{\beta}{(k - (k-1) \beta)}$ and so $\alpha$ and $\beta$ are either both computable or both non-computable.

\section{Kolmogorov Complexity}

Let us briefly recall some basics about Kolmogorov complexity (three good references on algorithmic complexity theory are \cite{LiVitanyi}, \cite{Nies09}, and \cite{DowneyHirschfeldt}). We call a partial computable function from $\fs$ to $\fs$ a \emph{machine} . For a machine~$M$, the Kolmogorov complexity relative to~$M$ is the function $C_M$ defined by $C_M(\sigma) = \min \{|p| : M(p) = \sigma\}$. There exist optimal machines which are machines~$\mathbb{U}$ such that for any machine~$M$, $C_\mathbb{U} \leq C_M + d$ for some constant $d$ (which depends on~$M$).

One can then fix an optimal machine $\mathbb{U}$ and define the Kolmogorov complexity of a string $\sigma$ to be $C_\mathbb{U}(\sigma)$. By definition of optimality, $C(\sigma)$ is independent of the choice of the optimal machine $\mathbb{U}$ up to an additive constant.

In the same vein, we can define conditional Kolmogorov complexity: the conditional Kolmogorov complexity of $\sigma$ given $\tau$, written $C(\sigma \mid \tau)$, is the length of the shortest program (or description) that produces~$\sigma$ when given~$\tau$ as input. Formally, given a partial computable function $M: \fs \times \fs \rightarrow \fs$, we define $C_M(\sigma \mid \tau) = \min \{|p| : M(p,\tau) = \sigma\}$. Again, it is easy to show that there exists a partial computable $\mathbb{V}: \fs \times \fs \rightarrow \fs$ such that for every other~$M$, $C_\mathbb{V}(\sigma \mid \tau) \leq C_M(\sigma \mid \tau) + d$ for some~$d$. Fixing such a~$\mathbb{V}$, we define $C(\sigma \mid \tau)=C_\mathbb{V}(\sigma \mid \tau)$.

Given a set $A$ of strings we can often make conclusions about the complexities of some or all of the members of $A$ based on the size of $A$. For example, if $A$ is large, then it must have a member of high complexity. The following fact is well-known and easy to see.

\begin{fact}\label{fact:measureC}
If $A \subset 2^{< \omega}$ is a set of strings $\sigma$ which each have $C(\sigma) \leq r$, then $| A|< 2^{r+1}$.
\end{fact}
\begin{proof}
Let $\mathbb{U}$ be the universal machine. There are at most $2^0 + 2^1 + 2^2 + \cdots + 2^r = 2^{r+1} - 1$ strings of length at most $r$ in the domain of $\mathbb{U}$, so $|A| < 2^{r+1}$.
\end{proof}

If $U$ is a small c.e.\ set, then the members of $U$ have low complexity. Moreover, the same is true for sequences of uniformly c.e.\ sets.
%We use the following standard proposition to prove two facts about such sequences.

%\begin{proposition}[Proposition 2.1.15 of \cite{Nies09}]
%Suppose that $D\colon \{0, 1\}^{< \omega} \to \mathbb{N} \cup \{\infty\}$ is computably approximable
%from above and satisfies the counting condition
%\[ |\{x : D(x) < k\}| < 2^k\]
%for each $k$.
%Then there is a machine $M$ such that for all $\sigma \in 2^{< \omega}$, $C_M(\sigma) = D(\sigma)+1$.
%\end{proposition}

\begin{proposition}\label{prop:small-ce-sets}
Let $(U_n)_{n \geq 1}$ be a sequence of uniformly c.e.\ finite sets.
Suppose that $|U_n| \leq 2^{k_n}$.
Then there is a constant $c$ such that for all $n$ and $\sigma \in U_n$, $C(\sigma) \leq k_n + 2 C(n \mid k_n) + c$. In particular, there is a $c'$ such that $C(\sigma) \leq k_n + 2 \log n + c'$.
\end{proposition}

\begin{proof}
Let $M$ be the machine which on an input $0^i1p$ starts by splitting $p$ as $p=qr$ with $|q|=i$. Then, it computes $n=\mathbb{V}(q,|r|-1)$. Finally, interpreting~$r$ as a natural number written in binary, it enumerates $U_n$ and returns the $r$-th enumerated element (if such an element is found). Now, if $\sigma$ is a member of $U_n$, since $|U_n| \leq  2^{k_n}$, one can write the index $r$ of $\sigma$ (in the order of the enumeration) in binary using~$k_n+1$ bits (padding with zeroes in front of this number if necessary). Then $|r|-1=k_n$, and if $q$ is the shortest $\mathbb{V}$-description of $n$ given $k_n$, we have $C_M(0^{|q|}1qr)=\sigma$ by construction, thus $C_M(\sigma) \leq k_n + 2 C(n \mid k_n) + 2$. The result follows by optimality of $\mathbb{V}$.
\end{proof}

%\begin{proposition}\label{prop:seq-ce-log}
%Let $(U_n)_{n \geq 1}$ be a sequence of uniformly c.e.\ sets.
%Suppose that $|U_n| \leq 2^{k_n}$.
%Then there is a constant $c$ such that for all $n$ and $\sigma \in U_n$, $C(\sigma) \leq k_n + 2 \log n + c$.
%\end{proposition}

\begin{corollary}\label{cor:seq-ceC}
Fix a computable $\alpha \in (0,1)$.
Let $(U_n)_{n \geq 1}$ be a sequence of uniformly c.e.\ sets.
Suppose that $|U_n| \leq 2^{\alpha n}$.
Then there is a constant $c$ such that for all $n$ and $\sigma \in U_n$, $C(\sigma) \leq \alpha n + c$.
\end{corollary}
\begin{proof}
Since $\alpha$ is computable, $n$ can be computed from $\lceil \alpha n \rceil$, which in particular implies $C(n \mid \lceil \alpha n \rceil)=O(1)$. The result then follows from Proposition~\ref{prop:small-ce-sets} with $k_n = \lceil \alpha n \rceil$.
\end{proof}

\section{\texorpdfstring{Characterization of $\EXT$}{The Total Case}}

In this section we will characterize the $(\alpha,\beta) \in \EXT(k)$. We begin in Section \ref{sec:total-one} by showing that when $\alpha, \beta$ are computable, $(\alpha,\beta) \in \EXT(k)$ is equivalent to the existence of a sequence of $k$-hypergraphs whose hyperedges are not too concentrated (in a sense determined by $\alpha$ and $\beta$) within any small set of vertices, thus translating our original problem into a purely combinatorial one. In Section \ref{sec:total-two}, we use the probabilistic method to construct such a sequence of hypergraphs for $\beta < \frac{k \alpha}{1 + (k-1)\alpha}$. Thus if $\beta < \frac{k \alpha}{1 + (k-1)\alpha}$ then $(\alpha,\beta) \in \EXT(k)$. In Section \ref{sec:total-three}, we show that if such a sequence of graphs exists then $\beta \leq \frac{k \alpha}{1 + (k-1)\alpha}$, and moreover if $k \geq 2$ and $\alpha,\beta \in (0,1)$, then $\beta < \frac{k \alpha}{1 + (k-1)\alpha}$. This completes the proof of Theorem \ref{thm:main-partial} together with the simple observation that if $\beta \leq \alpha$, then $(\alpha,\beta) \in \EXT(k)$ for any $k \geq 1$.

\subsection{Translating the problem: hypergraphs}\label{sec:total-one}

There are a number of different choices one may make when fixing the definition of a hypergraph, so in this section we will fix our definition for this paper. Our hypergraphs are $k$-uniform, undirected, and allow repeated hyperedges (so that two edges may be incident on the same set of vertices). All hyperedges are incident on exactly $k$ vertices. In this case $k = 2$, our $2$-hypergraphs are just undirected multigraphs which do not allow loops. More formally:

%\laurent{I am no longer convinced this definition makes more sense, because of pseudo-density, etc. We may want to just say that an edge is a k-uple with repetitions allowed. This way density and pseudo-density are the same.}

\begin{definition}
A $k$-hypergraph $G = (V,E)$ is a set of vertices $V$ and a set of hyperedges~$E$, with each edge $e \in E$ associated to a set $i(e)$ of $k$ vertices from $V$.
\end{definition}

In a graph, the edge density is the ratio of edges to potential edges. We make a similar definition here:

\begin{definition}
Let $G = (V,E)$ be a $k$-hypergraph. The edge pseudo-density $p$ of $G$ is
\[ p = \frac{|E|}{|V|^k}.\]
\end{definition}

\noindent The reason that we call this the edge pseudo-density rather than simply the edge density is that $|V|^k$ is slightly larger than ${|V| \choose k}$, the number of potential hyperedges, as hyperedges cannot have repeated vertices. Using $|V|^k$ rather than ${|V| \choose k}$ will make calculations easier.

Finally, given a set $U \subseteq V$ of vertices, we will want to consider the set of edges which are contained within $U$.

\begin{definition}
Let $G = (V,E)$ be a $k$-hypergraph, and let $U \subseteq V$. Then $E(U)$ is the set of edges which are incident only on vertices in $U$, and $e(U)$ is the cardinality of $E(U)$.
\end{definition}

\noindent This is the same as the set of edges in the sub-hypergraph induced by $U$.

The next lemma says that every hypergraph has a small sub-hypergraph with a similar (though possibly slightly smaller) edge pseudo-density. (If we used edge density instead, then we could get that the edge density does not decrease.)

\begin{lemma}\label{lem:subset-e(U)}
Fix $k \geq 2$. There is a constant $c_k$ such that for all~$n$, if $G = (V,E)$ is a $k$-hypergraph with $|V|=n$ and edge pseudo-density~$p$, then for any $c_k \leq u \leq n$ there exists a subset $U$ of $V$ of size $u$ such that $e(U) \geq 0.99 p u^k$ (or equivalently, $(U,E(U))$ has edge pseudo-density at least $0.99p$).
\end{lemma}

\begin{proof}
If we select the subset $U$ at random uniformly among subsets of $V$ of size~$u$, the probability that a fixed $k$-hyperedge $e \in E$ belongs to $E(U)$ is ${{n-k}\choose{u-k}}/{n\choose{u}}=\frac{u(u-1)\ldots(u-k+1)}{n(n-1)\ldots(n-k+1)}$. The numerator of this last expression is $\geq u^k (1-k/u)^k$, and the denominator is $\leq n^k$. Thus, the probability that a fixed edge~$e$ belongs to $e(U)$ is $\geq (u/n)^k \cdot (1-k/u)^k \geq (u/n)^k \cdot (1-k/c_k)^k$.

Since there are $pn^k$ edges in~$G$, this shows that
\[
\mathbb{E}(e(U)) \geq (u/n)^k (1-k/c_k)^k p n^k \geq p u^k (1-k/c_k)^k
\]
Thus, there must be some $U$ of size $u$ such that $e(U) \geq p u^k (1-k/c_k)^k$. It remains to choose $c_k$ large enough to have $(1-k/c_k)^k \geq 0.99$ to get the desired result.
\end{proof}

The next theorem allows us to convert the initial problem into a purely graph-theoretic one. The intuition is as follows. Suppose that we have functions $\Gamma_1,\ldots,\Gamma_k$ from $\{0,1\}^{f(n)}$ to $\{0,1\}^n$ which we want to have witness that $(\alpha,\beta) \in \EXT(k)$. We can think of our opponent as providing short descriptions for strings in $\{0,1\}^*$, trying to lower their Kolmogorov complexity. If, for some string $\sigma \in \{0,1\}^{f(n)}$, our opponent has provided short descriptions for $\Gamma_1(\sigma),\ldots,\Gamma_k(\sigma)$ (making them of dimension $< \beta$), we must provide a short description for $\sigma$ (making it of dimension $< \alpha$). Both our opponent and ourselves have some quantity of short descriptions that we can use, based on the values of $\alpha$ and $\beta$. We can think of a corresponding hypergraph, where the vertices are strings in $\{0,1\}^n$, and the hyperedges correspond to strings $\sigma \in \{0,1\}^{f(n)}$ which are incident on $\Gamma_1(\sigma),\ldots,\Gamma_k(\sigma)$. Our opponent is giving short descriptions to a set of vertices $U$ while we must give a short description to a hyperedge whenever our opponent gives a short description to every vertex on that hyperedge (i.e., we have to give short descriptions to each hyperedge in $E(U)$). Whether we or our opponent can win this game depends on the sizes of $U$ and $E(U)$ relative to the number of short descriptions we and our opponent have available.

\begin{theorem}\label{thm:graph-correspondence}
Fix $k \geq 2$ and let $(\alpha,\beta)$ be a pair of computable reals in $[0,1]$.  The following are equivalent
\begin{enumerate}
	\item[(a)] $(\alpha,\beta) \in \EXT(k)$
%	\item[(a)] There are $k$ total computable functions $\Gamma_1,\ldots,\Gamma_k$ such that:
%	\begin{itemize}	
%		\item if $|\sigma| = f(n)$ then $\Gamma_i(\sigma) = n$, and
%		\item there is a constant~$d$ s.t.\ if $|\sigma|=f(n)$ and $C(\sigma) \geq \alpha f(n) +d$, then $C(\Gamma_i(\sigma)) \geq \beta n - d$ for some~$i$.
%	\end{itemize}
	\item[(b)] There is a constant~$d$ and computable function~$f$ with $f(n) \geq (\beta/\alpha)n-O(1)$ and such that for all~$n$ there is a $k$-hypergraph $G_n$ with $2^n$ vertices and $2^{f(n)}$ hyperedges, with the property that for every $U \subseteq G_n$ with $|U| \leq 2^{\beta n-d}$, $e(U) < 2^{\alpha f(n)+d}$.
	\item[(c)] There is a constant~$d$ and computable function~$f$ with $f(n) \geq  (\beta/\alpha)n-O(1)$ and such that for all~$n$ there is a $k$-hypergraph $G_n$ with $2^n$ vertices and $2^{f(n)}$ hyperedges, with the property that for every $U \subseteq G_n$ with $|U| \leq 2^{\beta n}$, $e(U) < 2^{\alpha f(n)+d}$.
	
\end{enumerate}
\end{theorem}

\begin{proof}
$(a) \Rightarrow (b)$. Suppose~$(b)$ does not hold, and let us show that $(a)$ does not hold. Consider $k$ total computable functions $\Gamma_1,\ldots,\Gamma_k$ with $|\Gamma_i(\sigma)|=n$ when $|\sigma|=f(n)$. We can assume without loss of generality that for every $\sigma$, the $\Gamma_i(\sigma)$ are all different. Indeed, if this is not the case, we can replace the $\Gamma_i$ by the family $\Gamma'_i$ defined as follows: for all~$\sigma$, compute the set $A_\sigma=\{\Gamma_i(\sigma) \mid 1 \leq i \leq k\}$.  Since it has $ \leq k$ elements, computably find a finite set $B_\sigma \supset A_\sigma$ containing exactly $k$ elements, all of length~$n$ if $|\sigma|=f(n)$, and define $\Gamma'_i(\sigma)$ to be the $i$-th element of $B_\sigma$. The $\Gamma'_i$ are total, $\Gamma'_i(\sigma) \not=\Gamma'_j(\sigma)$ if $i \not= j$ and by construction for every $\sigma$ of length~$f(n)$, $\{\Gamma_i(\sigma) \mid 1 \leq i \leq k\} \subseteq \{\Gamma'_i(\sigma) \mid 1 \leq i \leq k\} \subseteq \{0,1\}^n$, From which it is easy to see that the $\Gamma'_i$ also witness that $(\alpha,\beta) \in \EXT(k)$.

Now, under this assumption that the $\Gamma_i(\sigma)$ are all different, for all~$n$, let~$G_n$ be the $k$-hypergraph whose set of vertices is the set of strings of length~$n$, and the hyperedges $e_\sigma$ are incident on $\Gamma_1(\sigma),\ldots,\Gamma_k(\sigma)$ where $\sigma$ ranges over strings of length~$f(n)$. Note that the sequence $(G_n)$ is computable.

By Remark~\ref{rem:lower-bound-for-f}, if $f(n) \not\geq (\beta/\alpha)n-O(1)$ then $(a)$ does not hold, so we may assume $f(n) \geq (\beta/\alpha)n-O(1)$. By failure of~$(b)$, for any given~$d$, there exists some~$n$ such that $G_n$ has a subset~$U$ of size $2^{\beta n-d}$ with $e(U) \geq 2^{\alpha f(n)+d}$. Since this is a decidable property (because $\alpha$, $\beta$ are computable), such a~$G_n$ and subset $U$ can be found effectively given~$d$. Thus, by Proposition \ref{prop:small-ce-sets}, for every $\tau \in U$, $C(\tau) \leq \log |U| + 2\log d +O(1)$, i.e., $C(\tau) \leq \beta n - d + 2\log d + O(1)$. On the other hand, since there are at least $2^{\alpha f(n)+d}$ many $\sigma$ with $e_\sigma$ in $E(U)$, by Fact \ref{fact:measureC} there must be one that satisfies $C(\sigma) \geq \alpha f(n) + d$. By definition of $e_\sigma$, we have that $\Gamma_i(\sigma) \in U$ for all~$i$, and so $C(\Gamma_i(\sigma)) \leq \beta n - d + 2\log d + O(1)$ for all~$i$. Since~$d$ can be taken arbitrarily large, this shows that $(a)$ fails.

\medskip{}

$(b) \Rightarrow (a)$. Fix a constant~$d$ and sequence $(G_n)$ of graphs witnessing that $(b)$ holds. The sequence $(G_n)$ can be taken to be computable as the property of having small $e(U)$ for all $U$ of size $2^{ \beta n - d}$ is decidable, so one can find the $G_n$ by exhaustive search. Then, for all~$n$, effectively create a bijection $\sigma \mapsto e_\sigma$ between strings of length~$f(n)$ and the hyperedges of $G_n$. Finally, for each $\sigma$, define $\Gamma_i(\sigma)$ for $i = 1,\ldots,k$ so that $e_\sigma$ is incident on $\Gamma_1(\sigma),\ldots,\Gamma_k(\sigma)$. The $\Gamma_i$ are total computable functions from strings of length~$f(n)$ to strings of length~$n$. Now, for each~$n$, consider the set $U \subseteq G_n$ of strings~$\tau$ such that $C(\tau) < \beta n-d$. Using Fact \ref{fact:measureC} we see that $|U| \leq 2^{\beta n-d}$, and so by property~$(b)$, $e(U) < 2^{\alpha f(n)+d}$. The sets $U$, and hence also the sets $E(U)$, are c.e.\ sets uniformly in~$n$. So by Corollary \ref{cor:seq-ceC} (and using the fact that the function $f$ is one-to-one) we have that $C(\sigma) \leq \alpha f(n)+d+O(1)$ for every $e_\sigma \in E(U)$. Taking the contrapositive, this means that when $C(\sigma) > \alpha f(n)+d+O(1)$, we have that $e_\sigma \notin E(U)$, which in turns means that some coordinate of $e_\sigma$ is not in~$U$, i.e., $C(\Gamma_i(\sigma)) \geq \beta n - d$ for some~$i$. This proves property~$(a)$.

\medskip{}

$(c) \Rightarrow (b)$. This is immediate.

\medskip{}

$(b) \Rightarrow (c)$. Let $(G_n)$ and $d$ be witnesses that $(b)$ holds. Let $c_k$ be the constant guaranteed by Lemma \ref{lem:subset-e(U)}. We may assume without loss of generality that $n$ is sufficiently large that $c_k \leq 2^{\beta n - d}$. Let $U$ be a subset of $G_n$ of with $|U| \leq 2^{\beta n}$. If $|U| \leq 2^{\beta n - d}$ then we are done. Otherwise, by Lemma~\ref{lem:subset-e(U)}, there exists a subset $U'$ of $U$ such that $|U'| = \left \lfloor{2^{\beta n -d}}\right \rfloor$ and
\[ e(U') \geq 0.99 \frac{e(U)}{|U|^2} |U'|^2 \geq 0.99 \frac{e(U)}{2^{2\beta n}} 2^{2 \beta n - 2d - 2} = 0.99 \cdot 2^{-2d - 2} e(U).\]
By $(b)$, we have $e(U') \leq 2^{\alpha f(n) + d}$. Putting the two together, we get $e(U) \leq 2^{\alpha f(n) +3d + O(1)}$. Thus $(c)$ holds as witnessed by the sequence $(G_n)$ and constant $3d + O(1)$.
\end{proof}

\subsection{The positive case: random hypergraphs}\label{sec:total-two}

Given $\beta < \frac{k \alpha}{1 + (k-1)\alpha}$, we want to show that $(\alpha,\beta) \in \EXT(k)$. By Theorem \ref{thm:graph-correspondence}, we can do this by constructing an appropriate sequence of hypergraphs. We will show that such a sequence exists using a probabilistic construction, i.e., by showing that if we choose a hypergraph at random, it has a positive probability of having the properties we want, and so, in particular, such a graph exists. In computing the associated probabilities, we will use the Chernoff bound. The Chernoff bound has many forms, and we state the two that we will use.

\begin{theorem}[Chernoff bound; see Theorem 4.4 (3) of \cite{MitzenmacherUpfal17}]
Let $X_1,\ldots,X_n$ be independent random variables taking values in $\{0,1\}$ and let $X$ be their sum. Let $\mu = \mathbb{E}[X]$.
\begin{enumerate}
\item[(1)]
For any $\delta \geq 6$,
\[ \Pr( X \geq \delta \mu) \leq 2^{- \delta\mu}.\]
\item[(2)]
For any $0\leq \delta \leq 1$,
\[ \Pr( X \leq (1-\delta) \mu) \leq e^{- \frac{\delta^2\mu}{2}}.\]
\end{enumerate}

\end{theorem}

\noindent We are now ready for the construction of the sequence of hypergraphs. One should think of taking $f(n) = [k - (k-1)\beta]n + O(1)$.

\begin{theorem}\label{thm:main-pos}
Fix $k$. Let $\alpha,\beta \in (0,1)$ be such that $\beta < \frac{k \alpha}{1 + (k-1)\alpha}$. There is a constant $d$ such that for each $n$ there is a $k$-hypergraph with $2^n$ vertices and at least $2^{[k - (k-1)\beta]n}$ hyperedges such that for every $U$ with $|U| \leq 2^{\beta n}$, $e(U) < 2^{\alpha [k - (k-1)\beta] n + d}$.
\end{theorem}

\begin{proof}
%By choosing $d$ to be large, it suffices to prove the theorem for $n$ sufficiently large. It also suffices to find a $k$-hypergraph graph $G$ whose hyperedges may have repetition among their vertices---for example, if $k = 3$, an edge may be incident on vertices $u$, $u$, and $v$---as we can replace such an edge with an edge incident on $u$, $v$, and some third vertex $w$. If the original graph satisfied the statement of the theorem, then so will the new graph.

We will show the existence of the graph $G$ by showing that a random graph is likely to satisfy the properties we desire. Consider picking a $k$-hypergraph $G$ with $2^n$ vertices at random, where each $k$-hyperedge has probability $p = 2^{-(k-1)\beta n+D}$ to belong to $G$, independently of other hyperedges, where $D$ is a large constant (to be specified as we go). The expected number of hyperedges in~$G$ is
\[
{2^n \choose k} \cdot p \geq  \left(2^{n-o(1)}\right)^k/k! \cdot 2^{-(k-1)\beta n+D} \geq 2^{[k - (k-1)\beta]n+D-o(1)}/k!
\]
Thus, by the Chernoff bound, if $D$ is chosen large enough, $G$ will have at least $2^{[k - (k-1)\beta]n}$ hyperedges (which is the desired amount), with probability $>1/2$.

Fix a set $U$ of at most $2^{\beta n }$ vertices. The expected number of hyperedges in $E(U)$ is thus $p$ times the number of sets of $k$ vertices in $U$, which gives
\[ \mathbb{E}[e(U)] \leq p {2^{\beta n} \choose k} \leq 2^{-(k-1)\beta n+D} \cdot (2^{\beta n})^k = 2^{\beta n+D}.\]
This is the case for all sets $U$ of vertices with $|U| \leq 2^{\beta n}$.

By the Chernoff bound,
\[ \Pr\left[e(U) > 2 n 2^{\beta n} \right] < 2^{- 2 n 2^{\beta n}}. \]
To use the Chernoff bound, we require $2 n 2^{\beta n} \geq 6 \mathbb{E}[e(U)]$ which it is easy to see is true for $n \geq 2^{D+2}$.
The number of sets $U$ of size at most $2^{\beta n}$ is less than $(2^n)^{2^{\beta n}} = 2^{n 2^{\beta n}}$.
So the probability that there is a set $U$ of size at most $2^{\beta n}$ with $e(U) > 2n2^{\beta n}$ is
\[ \sum_{|U|\leq 2^{\beta n}} \Pr\left[e(U) > 2 n 2^{\beta n} \right] < \sum_{|U|\leq 2^{\beta n}}  2^{- 2 n 2^{\beta n}} \leq 2^{n 2^{\beta n}} 2^{- 2 n 2^{\beta n}} = 2^{-n 2^{\beta n}}. \]
For sufficiently large $n$, this is strictly less than one half (which was the probability that $G$ had at least the desired number of edges).
So for sufficiently large $n$ there exists a graph $G$ with enough edges and such that for all $U$ with $|U| \leq 2^{\beta n}$, $e(U) \leq 2 n 2^{\beta n}$.
It remains to show that for sufficiently large $n$,
\[ 2 n 2^{\beta n} < 2^{\alpha[k - (k-1)\beta]n}. \]
We have that
\[ \beta < \frac{k \alpha}{1 + (k-1)\alpha} \]
and so
\[ \beta + (k-1) \beta \alpha < k \alpha \Longrightarrow \beta < k \alpha - (k-1) \beta \alpha = \alpha( k - (k-1) \beta). \]
It follows that, for sufficiently large $n$, for all sets $U$ of vertices from $G$ with $|U| \leq 2^{\beta n}$,
\[ e(U) \leq 2 n 2^{\beta n} < 2^{\alpha[k - (k-1)\beta]n}. \]
This completes the proof.
\end{proof}

\begin{corollary}
When $\beta < \frac{k \alpha}{1 + (k-1)\alpha}$, the pair $(\alpha,\beta)$ belongs to $\EXT(k)$.
\end{corollary}

\begin{proof}
We may assume that $\alpha,\beta$ are rational by replacing $\alpha$ by a rational $\alpha'<\alpha$ sufficiently close to $\alpha$ to have $\beta < \frac{k \alpha'}{1 + (k-1)\alpha'}$ and then a rational $\beta'$ between $\beta$ and $ \frac{k \alpha'}{1 + (k-1)\alpha'}$. If we can show that $(\alpha',\beta') \in \EXT(k)$, then it follows that $(\alpha,\beta) \in \EXT(k)$. So from now on, assume that $\alpha,\beta$ are rational.

Let $f(n) = \left \lfloor{(k -(k-1)\beta)n}\right \rfloor$; since $\beta$ is rational, this is computable. By Theorem~\ref{thm:main-pos} there is $d$ and a sequence $(G_n)$ of $k$-hypergraphs such that:
\begin{itemize}
	\item $G_n$ has $2^n$ vertices and at least $2^{[k - (k-1)\beta]n} \geq 2^{f(n)}$ hyperedges, and
	\item every set $U$ of vertices of $G_n$ with $|U| \leq 2^{\beta n}$ has $e(U) < 2^{\alpha [k - (k-1)\beta] n + d} < 2^{\alpha f(n) + (d + 1)}$.
\end{itemize}
Note that we may remove edges from $G_n$ so that it has exactly $2^{f(n)}$ edges while maintaining the other properties. By Theorem~\ref{thm:graph-correspondence} we have that $(\alpha,\beta)$ belongs to $\EXT(k)$.
\end{proof}

\subsection{The negative case: $\beta \geq k\alpha/(1+(k-1)\alpha)$}\label{sec:total-three}

In this section we will show that if $\beta > k\alpha/(1+(k-1)\alpha)$ then $(\alpha,\beta) \notin \EXT(k)$, and moreover, if $k \geq 2$, $\alpha,\beta \in (0,1)$, and $\beta = k\alpha/(1+(k-1)\alpha)$ then $(\alpha,\beta) \notin \EXT(k)$.

It is not hard to see what happens when $\beta > \frac{k\alpha}{1+(k-1)\alpha}$. Essentially, what happens is that the following proposition gives a lower and upper bound on $f(n)$ when $(\alpha,\beta) \in \EXT(k)$ (with the lower bound being that in Remark \ref{rem:lower-bound-for-f}), and then in the following corollary we see that the upper and lower bounds are incompatible when $\beta > k\alpha/(1+(k-1)\alpha)$.

\begin{proposition}\label{prop:controlling-f}
Suppose $\alpha, \beta$ are computable and $(\alpha,\beta)$ belongs to $\EXT(k)$. By Theorem~\ref{thm:graph-correspondence}, let~$d$ be a constant and computable function~$f$ such that $f(n) \geq (\beta/\alpha)n - O(1)$ and $(G_n)$ a sequence of hypergraphs where $G_n$ has $2^n$ vertices, $2^{f(n)}$ hyperedges and the property that for every $U \subseteq G_n$ with $|U| \leq 2^{\beta n}$, $e(U) < 2^{\alpha f(n)+d}$. Then the following inequality holds:
\[ (\beta/\alpha) n - O(1) \leq f(n) \leq \frac{1-\beta}{1-\alpha}kn +O(1),\]
where the $O(1)$ on the right hand side is dependent on $\alpha$ and $d$ as well as $k$.
\end{proposition}

\begin{proof}
$(\beta/\alpha) n - O(1) \leq f(n)$ is part of the assumption on~$f$, so we only need to prove $f(n) \leq kn(1-\beta)/(1-\alpha) +O(1)$.

For all~$n$, the edge pseudo-density of $G_n$ is equal to $p=2^{f(n)}/2^{kn}=2^{f(n)-kn}$. Let $n$ be sufficiently large. By Lemma~\ref{lem:subset-e(U)}, there is a subset $U$ of vertices of $G_n$ such that $2^{\beta n - 1} \leq |U| = 2^{\beta n}$ and $e(U) \geq 0.99 p (2^{\beta n - 1})^k = 0.99\cdot 2^{f(n)-kn} 2^{k \beta n-k} \geq 2^{f(n)-kn + k\beta n-k - 1}$. By assumption on $G_n$, we also have $e(U) < 2^{\alpha f(n) +d}$. Thus:
\[
f(n)-kn + k\beta n - k - 1 < \alpha f(n) + d
\]
This can be rewritten as
\[
f(n) < \frac{kn(1-\beta)}{1-\alpha} + \frac{d+k+1}{1 - \alpha}
\]
as desired.
\end{proof}

As a direct corollary, we get:

\begin{corollary}
If $\beta > k\alpha/(1+(k-1)\alpha)$, then $(\alpha,\beta) \not\in \EXT(k)$.
\end{corollary}

\begin{proof}
Let $\beta'<\beta$ and $\alpha'>\alpha$ be rationals such that $\beta' > k\alpha'/(1+(k-1)\alpha')$.
The inequality $\beta' > k\alpha'/(1+(k-1)\alpha')$ is equivalent, mutatis mutandis, to $(\beta'/\alpha')>k(1-\beta')/(1-\alpha')$. Therefore, there cannot be a function $f$ such that $(\beta'/\alpha') n - O(1) \leq f(n) \leq kn(1-\beta')/(1-\alpha') +O(1)$, which by Proposition~\ref{prop:controlling-f} shows that $(\alpha',\beta') \notin \EXT(k)$. Since $\alpha > \alpha'$ and $\beta < \beta'$, this shows a fortiori that $(\alpha,\beta) \notin \EXT(k)$.
\end{proof}

The last case we need to treat, which turns out to be more difficult, is when $\beta = k\alpha/(1+(k-1)\alpha)$. In this case, for $k = 1$, we get $\alpha = \beta$ in which case $(\alpha,\beta) \in \EXT(1)$ as witnessed by $\Gamma$ being the identity. For $k \geq 2$, if $\alpha = 1$ or if $\beta = 0$ then taking $\Gamma_1$ to be the identity also works. So we are left with the case $k \geq 2$ and $\alpha,\beta \in (0,1)$. In this case, we will show that  $(\alpha,\beta) \not\in \EXT(k)$. We first prove this result for $\alpha, \beta$ computable, and -- using a different method -- will deal with the case $\alpha, \beta$ uncomputable in the next section (Theorem~\ref{thm:threshold-not-computable}).

\begin{theorem}\label{thm:neg-tot}
Let $k \geq 2$ and suppose that $\alpha,\beta \in (0,1)$  are computable.  If $\beta = k\alpha/(1+(k-1)\alpha)$, then $(\alpha,\beta) \not\in \EXT(k)$.
\end{theorem}
\begin{proof}
For the sake of contradiction, assume that $(\alpha,\beta) \in \EXT(k)$, and let~$d$ be a constant, $f$ a computable function such that $f(n) \geq (\beta/\alpha)n - O(1)$, and $(G_n)$ a sequence of hypergraphs where $G_n$ has $2^n$ vertices, $2^{f(n)}$ hyperedges and the property that for every $U \subseteq G_n$ with $|U| \leq 2^{\beta n}$, $e(U) < 2^{\alpha f(n)+d}$. By Proposition~\ref{prop:controlling-f}, we must have $(\beta/\alpha) n - O(1) \leq f(n) \leq kn(1-\beta)/(1-\alpha) +O(1)$, but $\beta = k\alpha/(1+(k-1)\alpha)$ implies $\beta/\alpha=k(1-\beta)/(1-\alpha)$, so we get a precise expression for the function~$f$, namely
\begin{equation}
f(n) =(\beta/\alpha)n + O(1)=(k-(k-1)\beta)n + O(1)=(k/(1+(k-1)\alpha))n +O(1).
\end{equation}
The $O(1)$ depends on $\alpha$ and $d$ as well as $k$.
From this, we can rewrite the property of $G_n$ as follows, for a possibly different value of $d$:
\begin{equation}
\text{for every~ } U \subseteq G_n \text{~ with~ } |U| \leq 2^{\beta n},~ e(U) < 2^{\beta n+d}.
\end{equation}

Note that if $k = 1$, then such a sequence of graphs $G_n$ does exist. The key to finish the proof is the following combinatorial lemma which says that such a sequence does not exist for $k \geq 2$.

\begin{lemma}\label{lem:seq}
Let $k \geq 2$ and $\beta \in (0,1)$. Let $(G_n)$ be a sequence of $k$-hypergraphs such that $G_n$ has $2^n$ vertices and $2^{[k-(k-1)\beta]n-O(1)}$ hyperedges. For any constant~$D$, there is an~$n$ and a subset $U$ of vertices of $G_n$ with $|U| \leq 2^{\beta n}$ and $e(U) \geq 2^{\beta n + D}$.
\end{lemma}

\begin{proof}
We prove this result by induction over~$k$.\\

\noindent \textbf{Base case: $k=2$}. In this case the $G_n$ are just binary multigraphs, but remember that there can be multiple edges between two vertices. Fix a constant~$D$. We begin by removing some edges from the $G_n$ to give them a simpler structure while preserving the hypotheses of the theorem. For each pair $\{x,y\}$ of vertices of $G_n$, recall that $e(\{x,y\})$ is the number of edges between $x$ and $y$. Let $P$ be the set of pairs $\{x,y\}$ that have the $\left\lfloor{2^{\beta n-1}}\right\rfloor$ biggest values of $e(\{x,y\})$, and $U_{max} = \bigcup_{\{x,y\} \in P} \{x,y\}$. Note that $|U_{max}| \leq 2|P| \leq 2^{\beta n}$. If $e(U_{max}) = \sum_{\{x,y\} \in P} e(\{x,y\})$ is greater or equal to $2^{\beta n +D}$, we are done, so we may assume this quantity to be  $< 2^{\beta n +D}$. Observe that this means that $\sum_{\{x,y\} \in P} e(\{x,y\}) < 2^{\beta n +D}$, so by the pigeonhole principle, there is some $\{x,y\} \in P$ such that $e(\{x,y\}) < 2^{\beta n +D}/ \left\lfloor{2^{\beta n-1}}\right\rfloor \leq 2^{D+2}$. By definition of~$P$, this shows that $e(\{x,y\}) < 2^{D+2}$ for any $\{x,y\} \notin P$.

Now we remove from $G_n$ the edges in $E(U_{max})$, and the resulting multigraph will still have at least $2^{(2-\beta)n-O(1)} - 2^{\beta n +D}$ edges, which is still $2^{(2-\beta)n-O(1)}$ since $\beta < 1$. The $O(1)$ constant depends on~$D$, but this will not cause any problems.

Moreover, as we saw, between any two vertices in the resulting graph there are at most $2^{D+2}$ edges. So we may collapse all edges between any pair of vertices into one edge, which will divide the number of edges by at most $2^{D+2}$, and thus we will still have $2^{(2-\beta)n-O(1)}$ edges in the resulting graph, which will now have at most one edge between any two vertices. Thus we have obtained a graph rather than a multigraph.

Next, we make the graph bipartite with two sides of equal size, while keeping at least $1/5$ of the edges. This can be done because if we choose a partition of the vertices into two sets of size $2^{n-1}$ at random among all partitions, the probability for a given edge to have one coordinate on each side is $1/4-o(1)$. Thus, there must exist some fixed partition which splits the graph into two parts and has the property that a fraction $1/4 - o(1)$ of the edges go from one side to the other. We remove from our graph the edges which do not have a coordinate on each side. The graph is now bipartite and still has $2^{(2-\beta)n-O(1)}$ edges.

We have thus obtained a new sequence of subgraphs $G'_n$ of $G_n$ where $G'_n$ has the same vertices as $G_n$, $2^{(2-\beta)n-O(1)}$ edges, has at most one edge between any two vertices, and is bipartite with two sides $L_n$ and $R_n$ (for `left' and `right') of size $2^{n-1}$ each. We will now try to find some~$n$ and some subset $U$ of vertices of $G'_n$ of size at most $2^{\beta n}$ and such $e(U)\geq 2^{\beta n +D}$ (inside $G'_n$, and thus inside $G_n$ as well). From now on, we work inside the~$G'_n$.

For all~$n$, we need to distinguish two cases, corresponding to whether or not a lot of edges are concentrated on a small amount of vertices. For all~$n$, let $A_n$ be the set of $\left\lfloor{2^{\beta n-1}}\right\rfloor$ vertices~$x$ in $L_n$ that have the largest values $e(x,R_n)$. Our two cases are as follows.

\medskip{}

\textit{Case 1: $\sum_{x \in A_n} e(x,R_n) \geq 2^{n + D + 1}$}. In this case, we claim that there is $B_n \subseteq R_n$ of size at most $2^{\beta n -1}$ such that $e(A_n,B_n) \geq 2^{\beta n + D}$. Indeed, let $B_n \subseteq R_n$ be the $\left\lfloor{2^{\beta n - 1}}\right\rfloor$ nodes $y$ from $R_n$ with the largest values of $e(A_n,y)$. We have that $\sum_{y \in R_n} e(A_n,y) = \sum_{x \in A_n} e(x,R_n) \geq 2^{n + D + 1}$, and so
\[ e(A_n,B_n) \geq \left\lfloor{2^{\beta n - 1}}\right\rfloor \frac{2^{n + D + 1}}{2^{n-1}} \geq 2^{\beta n + D}.\]
This is what we wanted.

\medskip{}

\textit{Case 2: $\sum_{x \in A_n} e(x,R_n) < 2^{n + D + 1}$}. Our first step is to find a large subset $Q_n$ of $L_n$ such that for each $x \in Q_n$, $e(x,R_n)$ is reasonably large. To begin, note that there must be some $x \in A_n$ such that $e(x,R_n) < 2^{n+D-\beta n + 3}$. By definition of $A_n$, this implies
\begin{equation}\label{eq:Ln-An}
e(x,R_n) < 2^{(1-\beta) n +D + 3}~ \text{ for all ~ }  x \in L_n \setminus A_n.
\end{equation}
Note that $|L_n \setminus A_n| \geq 2^{n-1}-2^{\beta n -1} \geq 2^{n-2}$, and $\sum_{x \in L_n \setminus A_n} e(x,R_n) \geq 2^{(2-\beta)n - O(1)} - 2^{n + D + 1}$ so, calling $\delta(n)$ the average value of $e(x,R_n)$ over $x \in L_n \setminus A_n$, we have $\delta(n)=2^{(1-\beta)n-O(1)}$. Here, and for the remainder of this base case, $O(1)$ will depend on $D$.

Let $Q_n = \{x \in L_n \setminus A_n \mid e(x,R_n) \geq \delta(n)/2\}$. We claim that
\begin{equation}
|Q_n| \geq 2^{n - O(1)}
\end{equation}
Indeed,
\begin{eqnarray*}
\sum_{x \in L_n \setminus A_n} e(x,R_n) & \leq & \Big(|L_n \setminus A_n| - |Q_n|\Big) \delta(n)/2 + |Q_n| \cdot 2^{(1-\beta) n + O(1)} \\
 & \leq & |L_n \setminus A_n| \cdot \delta(n)/2 + |Q_n| \cdot 2^{(1-\beta) n+O(1)}
 \end{eqnarray*}
(the first inequality is a consequence of~\eqref{eq:Ln-An}), and since $\sum_{x \in L_n \setminus A_n} e(x,R_n) = |L_n \setminus A_n| \cdot \delta(n)$ (by definition of $\delta(n)$), this yields
\begin{eqnarray*}
|Q_n| & \geq & 2^{(\beta - 1)n - O(1)} \cdot \frac{1}{2} \sum_{x \in L_n \setminus A_n} e(x,R_n) \\
 & \geq & 2^{(\beta - 1)n - O(1)} \cdot 2^{(2-\beta) n - O(1)}\\
 & \geq & 2^{n - O(1)}
\end{eqnarray*}
 as desired.

 Suppose now that we were to choose a subset $B \subseteq R_n$ at random by putting each $y \in R_n$ into~$B$ with probability $2^{(\beta-1) n - 3}$ independently of the other vertices of $R_n$. The expected value of $|B|$ is $2^{n-1} \cdot 2^{(\beta-1) n - 3}=2^{\beta n -4}$. The Chernoff bound shows that
 \[
 \mathbb{P}\Big(|B| \geq 2^{\beta n -1}\Big) < 2^{-2^{\beta n -1}}
 \]
for sufficiently large $n$. In particular,  $\mathbb{P}\Big(|B| < 2^{\beta n -1}\Big) = 1-o(1)$.
 Furthermore, consider a fixed $x \in Q_n$. Recall that this means $e(x,R_n) = 2^{(1-\beta) n- O(1)}$ (the $O(1)$ constant depending on~$D$). The key point is to evaluate the distribution of $e(x,B)$ when $B$ is chosen randomly. For this, we use the Poisson limit theorem (a.k.a.\ law of rare events):

 \begin{theorem}[Law of rare events]
If we have $N$ $\{0,1\}$-valued independent random variables $X_1, \ldots, X_N$ where $X_i$ is equal to $1$ with probability $\lambda/N$, then the distribution of $\sum_i X_i$ converges, as $N \rightarrow \infty$, to the Poisson distribution of parameter $\lambda$ (which is the distribution over $\N$ where $K$ has probability $(\lambda^K e^{-\lambda}) / K!$).
 \end{theorem}

This is exactly the situation of $e(x,B)$, which is the sum of $2^{(1-\beta)n - O(1)}$ binary random variables (whether or not each of the edges emanating from $x$ will have their other vertex included in $B$), each of which has probability $2^{(\beta - 1)n - O(1)}$ to be equal to~$1$. So we have $N = 2^{(1-\beta)n - O(1)}$ and $\lambda = 2^{\Omega(1)} > 0$. Therefore, for sufficiently large $n$, there is an $\epsilon > 0$ such that
 \[
 \mathbb{P}\Big(e(x,B) \geq 2^{D+1}\Big) \geq \mathbb{P}\Big(e(x,B) = 2^{D+1}\Big) = (\Omega(1) e^{-\Omega(1)}) / 2^{D+1} > \varepsilon
 \]
 Thus, when $B$ is chosen randomly as above, the expected value of $|\{x \in Q_n \mid e(x,B) \geq 2^{D+1}\}|$ is $\geq \varepsilon |Q_n| \geq \varepsilon \cdot 2^{n-O(1)}$. For $n$ large enough, this is greater than $2^{\beta n -1}$ as $\beta < 1$, and so for $n$ large enough, there exists a set $B_n \subseteq R_n$ of size $\left \lfloor {2^{\beta n-1}}\right\rfloor$ such that
 \[
 |\{x \in Q_n \mid e(x,B_n) \geq 2^{D+1}\}| \geq 2^{\beta n -1}
 \]
Thus, we can take a subset $Q'_n$ of $Q_n$ of size $\left\lfloor{2^{\beta n -1}}\right\rfloor$ such that $e(x,B_n) \geq 2^{D+1}$ for all~$x \in Q'_n$, and set $U=Q'_n \cup B_n$. We have $|U|\leq2^{\beta n-1}+2^{\beta n-1}=2^{\beta n}$ and $e(U) \geq 2^{D+1} |Q'_n| \geq 2^{\beta n +D}$. This is what we wanted.

This concludes the base case $k=2$.

\medskip{}

\noindent \textbf{Induction step}. Suppose now $k>2$ and that the theorem holds for $k-1$. We have a sequence of $k$-hypergraphs $(G_n)$ where $G_n$ has $2^n$ vertices and $2^{[k-(k-1)\beta]n-O(1)}$ hyperedges, and we fix a large constant $D$.

To reduce the problem to $(k-1)$-hypergraphs, we once again use the probabilistic method. For each~$n$, if we select at random a set $A$ of size $\left\lfloor{2^{\beta n}}\right\rfloor$, and let $F$ be the set of hyperedges that have at least one component in~$A$, the probability that a given hyperedge of $G_n$ belongs to $F$ is, for $n$ much larger than $k$,
\begin{eqnarray*}
1 - \frac{{2^n-k \choose{\left\lfloor{2^{\beta n}}\right\rfloor}}}{{2^n \choose{\left\lfloor{2^{\beta n}}\right\rfloor}}} &=&
1 - \frac{(2^n - k)! (2^n - \left\lfloor{2^{\beta n}}\right\rfloor)!}{(2^n)!(2^n - k - \left\lfloor{2^{\beta n}}\right\rfloor)!}\\ &=&
1 - \frac{2^n - k}{2^n} \cdots \frac{2^n - k - \left\lfloor{2^{\beta n}}\right\rfloor + 1}{2^n - \left\lfloor{2^{\beta n}}\right\rfloor + 1} \\&=&
1 - \frac{(2^n - \left\lfloor{2^{\beta n}}\right\rfloor) \cdots (2^n - k - \left\lfloor{2^{\beta n}}\right\rfloor + 1)}{(2^n) \cdots (2^n - k + 1)} \\ & \geq &
1 - \left(\frac{2^n - 2^{\beta n - 1}}{2^n}\right)^k \\
&=& 1 - (1 - 2^{(\beta -1)n - 1})^k \\
&\geq& \frac{k}{2} 2^{(\beta - 1) n} - O\left(2^{2(\beta - 1)n}\right)\\
&\geq& \frac{k}{4} 2^{(\beta - 1) n}.
\end{eqnarray*}
We use the fact that $n$ is much larger than $k$ in the first line and in the last two lines. Thus
\[
\mathbb{E}(|F|) \geq \frac{k}{4}  \cdot 2^{(\beta - 1)n} \cdot 2^{[k-(k-1)\beta]n-O(1)} =  2^{[(k-1)-(k-2)\beta]n-O(1)}
\]
We can therefore choose for each~$n$ a subset $A_n$ of size $\left\lfloor{2^{\beta n}}\right\rfloor$ such that the corresponding sequence of $F_n$ is such that $|F_n|=2^{[(k-1)-(k-2)\beta]n-O(1)}$.

Now, for each~$n$, for each $k$-hyperedge $e\in F_n$, consider the $(k-1)$-hyperedge $e'$ obtained by removing from $e$ the coordinate that belongs to $A_n$, or one of those coordinates if there are several. Let $F'_n$ be the set of $(k-1)$-hyperedges obtained in this fashion. This operation does not change the cardinality so $|F'_n|=2^{[(k-1)-(k-2)\beta]n-O(1)}$. Let $H_n$ be the $(k-1)$-hypergraph whose vertices are the same as those of $G_n$ and whose set of hyperedges is~$F'_n$.

We can now apply our induction hypothesis at level $(k-1)$ to the sequence $(H_n)$ and constant $(D+k+1)$, to get some~$n$ and some subset $B_n$ of vertices of $H_n$ such that $|B_n| \leq 2^{\beta n}$ and $e_{H_n}(B_n) \geq 2^{\beta n + D + k+1}$.

Observe that $e_{G_n}(A_n \cup B_n) \geq e_{H_n}(B_n)$. Indeed, if a $(k-1)$-hyperedge $e' \in F'_n$ has all its coordinates in $B_n$, the $k$-hyperedge $e$ of $G_n$ it came from has $(k-1)$ coordinates in $B_n$, and one coordinate in $A_n$, hence all its coordinates are in $A_n \cup B_n$. Thus $e_{G_n}(A_n \cup B_n) \geq 2^{\beta n + D+k+1}$. And since $|A_n \cup B_n| \leq |A_n| + |B_n| \leq 2^{\beta n +1}$, by Lemma~\ref{lem:subset-e(U)}, there is a subset $U$ of $A_n \cup B_n$ of size $|A_n \cup B_n|/2 \leq 2^{\beta n}$ such that $e_{G_n}(U) \geq 0.99 \cdot 2^{-k} \cdot e_{G_n}(A_n \cup B_n) \geq 2^{\beta n +D}$. The set~$U$ is as wanted, and this concludes the induction step.

This completes the proof of Lemma \ref{lem:seq} and thus of the theorem.
\end{proof}
\renewcommand{\qedsymbol}{}\end{proof}

\section{\texorpdfstring{Characterization of $\EXTp$}{The Partial Case}}

For the partial case, we immediately inherit all of the positive results from the total case as $\EXT(k) \subseteq \EXTp(k)$. It is not hard to see that $\EXT(1)$ and $\EXTp(1)$ are the same. To see this, it suffices to show that $\EXTp(1) \subseteq \EXT(1)$ as we already know that $\EXT(1) \subseteq \EXTp(1)$. If $(\alpha,\beta) \in \EXTp(1)$ as witnessed by $\Gamma$, $f$, and $d$, with $\Gamma$ partial, then define $\Phi(\sigma)$ to be either $\Gamma(\sigma)$ or the all zeros string, depending on whether we find out first that $\Gamma(\sigma)$ converges or that $C(\sigma) < \alpha|\sigma| + d$. Note that $\Phi$ is total as $\Gamma$ is defined on all $\sigma$ with $C(\sigma) \geq \alpha|\sigma| + d$ and so witnesses that $(\alpha,\beta) \in \EXT(1)$. So for the remainder of this section, we can consider only the case $k \geq 2$.

In the previous section, we showed that if $(\alpha,\beta) \in \EXT(k)$ as witnessed by $f$, then
\[ (\beta/\alpha) n - O(1) \leq f(n) \leq \frac{1-\beta}{1-\alpha}kn +O(1) \]
and moreover, that $\beta / \alpha \leq \frac{1-\beta}{1-\alpha}k$ was equivalent to $\beta \leq \frac{k \alpha}{1+(k-1)\alpha}$. As $\EXT(k) \subset \EXTp(k)$, we know that $(\alpha,\beta) \in \EXTp(k)$ when $\beta < \frac{k \alpha}{1+(k-1)\alpha}$, when $\beta = 0$, or when $\alpha = 1$. In this section we consider the case when $\beta \geq \frac{k \alpha}{1+(k-1)\alpha}$ to see if any such pairs $(\alpha, \beta)$ belong to $\EXTp(k)$. We will show (Theorem \ref{prop:sqrt-deviation}) that for such $(\alpha, \beta)$, if $(\alpha,\beta) \in \EXTp(k)$, then
\[
(\beta/\alpha) n - O(1) \leq f(n) \leq \frac{k}{1+(k-1)\alpha} n + \sqrt{n} + O(1).
\]
From this we get that $\beta \leq \frac{k \alpha}{1+(k-1)\alpha}$, and so (Corollary \ref{cor:partial-not}) if $\beta > \frac{k \alpha}{1+(k-1)\alpha}$ then $(\alpha,\beta) \notin \EXTp(k)$.

This leaves the case $\beta = \frac{k \alpha}{1+(k-1)\alpha}$. This case will depend on whether or not $\alpha$ and $\beta$ are computable.

To prove Theorem \ref{prop:sqrt-deviation}, we will use the following lemma.

\begin{lemma}\label{lem:negative-partial-2}
Let $(D_n)$ be a computable sequence of finite sets of strings, and $\Gamma_1, \ldots, \Gamma_k$ be partial computable functions from $\fs$ to $\fs$, such that $\Gamma_i(D_n) \subseteq \{0,1\}^n$ for all $i,n$. Let $\varphi: \N \rightarrow \N$ be a function such that $\varphi(n) \leq n$ for all~$n$ (we do not assume that $\varphi$ is computable). There is a constant $d$ such that for all~$n$, there is some $x \in D_n$ such that:
\begin{itemize}
	\item $C(x) > \log|D_n| -(n-\varphi(n)+1)k - d$, and
	\item for every $i \leq k$, either $\Gamma_i(x)$ is undefined, or $C(\Gamma_i(x)) < \varphi(n) + 2C(\varphi(n),n)+d$.
\end{itemize}
\end{lemma}

\begin{proof}
We will show that there is a subset $E_n$ of $D_n$ such that
\begin{itemize}
\item $|E_n| \geq |D_n| \cdot 2^{-(n-\varphi(n)+1)k}$
\item for every $x \in E_n$, for every $i \leq k$, either $\Gamma_i(x)$ is undefined, or $C(\Gamma_i(x)) < \varphi(n) + 2C(\varphi(n),n)+d$.
\end{itemize}
Then, since any set of string of cardinality $\geq 2^s$ contains an element of Kolmogorov complexity at least $s$ (Fact \ref{fact:measureC}), there is $x \in E_n$ with $C(x) > \log|D_n| -(n-\varphi(n)+1)k - d$.

The functionals $\Gamma_1, \ldots, \Gamma_k$ play symmetric roles, so we can assume that for all $x$ and $i<j$, $\Gamma_j(x)$ can only converge if $\Gamma_i(x)$ does. Indeed, let $\Gamma'_i(x)$ be the $i$-th element that appears in the uniform enumeration of the c.e.\ set $\{\Gamma_i(x) \mid i \leq k\}$, if such an element appears. The $\Gamma'_i$ are as desired and replacing each $\Gamma_i$ by $\Gamma'_i$ does not change the truth value of the statement of the proposition.

Let us now fix an~$n$. Consider the following algorithm, which is uniform in $n$ and $\varphi(n)$, but not necessarily in $n$ alone. Set $A_0=D_n$. For every $i$ from $1$ to~$k$, do the following:
\begin{enumerate}
\item Enumerate $\dom(\Gamma_i) \cap A_{i-1}$ until we see at least $|A_{i-1}|/2$ elements being enumerated. If this happens, move on to Step 2 (otherwise we wait forever at this stage).
\item Let $\Phi_i$ be the (total) restriction of $\Gamma_i$ to these $\geq |A_{i-1}|/2$ elements of $A_{i-1}$.
\item Let $B_i$ be the set consisting of the $2^{\varphi(n)}$ strings~$y \in \{0,1\}^n$ that have the largest $2^{\varphi(n)}$ values of $|\Phi_i^{-1}(y)|$ among strings of length~$n$.
\item Set $A_i = \Phi_i^{-1}(B_i)$.
\item If $i<k$, increase $i$ by $1$ and start the loop again.
\end{enumerate}

Let $j$ be the index of the last loop that is completed, and let $B = \bigcup_{i=1}^j B_i$. $B$ is c.e.\ uniformly given $n$ and $\varphi(n)$ as parameters. Let us make several easy observations about the sets $A_i$ and $B_i$.
\begin{itemize}
	\item By construction, $A_0 \supseteq A_1 \supseteq \ldots \supseteq A_j$.
	\item Again by construction, $\Gamma_i(A_i) \subseteq B_i$ for all $i$, so $\Gamma_i(A_j) \subseteq B_i$ for all $i$, which in turn implies $\Gamma_i(A_j) \subseteq B$ for all~$i$.
	\item Each set $B_i$ has cardinality $2^{\varphi(n)}$, so $B$ has cardinality at most $k \cdot 2^{\varphi(n)}$.
	\item For all $i>0$, we have $|A_i| \geq 2^{\varphi(n)-n-1} \cdot |A_{i-1}|$ when $A_i$ is defined. Indeed, $\Phi_i$ is a function from a set of size at least $|A_{i-1}|/2$ to a set of size $2^n$, so the average value of $|\Phi_i^{-1}(y)|$ is at least $2^{-n} \cdot |A_{i-1}|/2$. If we take the $2^{\varphi(n)}$ greatest such values, their sum, which is the cardinality of $A_i$ by definition, is at least $2^{\varphi(n)} \cdot 2^{-n} \cdot |A_{i-1}|/2$, as desired. By induction, this tells us that $|A_i| \geq 2^{(\varphi(n)-n-1)i}|A_0|$ when $A_i$ is defined.
\end{itemize}
Let us now build the advertised set $E_n$. There are two cases. If $j=k$ (all loops of the algorithm are performed), simply let $E_n=A_j$. If $j<k$, let $E_n = A_j \setminus \dom(\Gamma_{j+1})$. Note that the set $E_n$ is not computable or even c.e.\ in $n$, but this will not matter.

In the first case, we have $|E_n|=|A_k| \geq 2^{(\varphi(n)-n-1)k}|A_0|$ by the above calculation, and since $A_0=D_n$ this is what we want. In the second case ($j<k$), by definition of~$j$, the algorithm must get stuck at Step 1 of the $j+1$-th loop, that is, we must have $|\dom(\Gamma_{j+1})\cap A_j| \leq |A_j|/2$, so $|E_n| \geq |A_j|/2 \geq 2^{(\varphi(n)-n-1)j-1}|A_0| \geq 2^{(\varphi(n)-n-1)k}|A_0|$ (for the last inequality, we use the fact that $j<k$ and $\varphi(n) \leq n$).

So in either case, we have
\[
|E_n|\geq 2^{(\varphi(n)-n-1)k}|D_n|
\]
Moreover, the definition of $E_n$ ensures that for any $x \in E_n$, $\Gamma_i(x)$ is defined and belongs to $B$ if $i \leq j$, and $\Gamma_{j+1}(x)$ is undefined (thus, by our initial assumption of the $\Gamma_i$, $\Gamma_i(x)$ is undefined for every $i \geq j+1$). These two facts together imply $\Gamma_i(E_n) \subseteq B$ for all~$i$. To complete the proof, observe that the construction is effective (given $n$ and $\varphi(n)$ as parameters), so the set $B$ can be uniformly enumerated if $n$ and $\varphi(n)$ are known. Since $|B| \leq k \cdot 2^{\varphi(n)}$, $C(y|\varphi(n),n) < \varphi(n) + \log k + O(1)$ for all $y \in B$. Using the fact that $C(y|u) > C(y) -2C(u) - O(1)$ for all $y,u$, we get that for every~$y \in B$, $C(y) < \varphi(n) + 2C(\varphi(n),n)+O(1)$.
\end{proof}

We are now ready to prove our bound on the functions $f$ which can witness that $(\alpha,\beta) \in \EXTp(k)$.

\begin{theorem}\label{prop:sqrt-deviation}
Let $k \geq 1$, $\beta \geq \frac{k \alpha}{1+(k-1)\alpha}$, and suppose $(\alpha,\beta)$ belongs to $\EXTp(k)$. Let $f$ be a computable function witnessing this. Then the following inequality holds:
\[ (\beta/\alpha) n - O(1) \leq f(n) \leq \frac{k}{1+(k-1)\alpha} n + \sqrt{n} + O(1),\]
where the $O(1)$ on the right hand side is dependent on $\alpha$ and $d$ as well as $k$.
\end{theorem}

The choice of $\sqrt{n}$ is somewhat arbitrary; all we need for the proof is a computable function $h$ such that $\log n = o(h(n))$, but fixing $h(n)=\sqrt{n}$ is sufficient for our purposes.

\begin{proof}
The inequality $(\beta/\alpha) n - O(1) \leq f(n)$ is from Remark \ref{rem:lb2}.
To prove the second inequality, we will show the contrapositive.
Suppose that $(\beta/\alpha) n - O(1) \leq f(n)$ but that
\[ f(n) \nleq \frac{k}{1+(k-1)\alpha} n + \sqrt{n} + O(1).\]
Then for infinitely many $n$,
\[ f(n) \geq \frac{k}{1+(k-1)\alpha} n + \sqrt{n} .\]
%We will show that $(\alpha,\beta) \notin \EXTp(k)$.

Let $\hat{\alpha} = \frac{k \alpha}{1+(k-1)\alpha}$. Note that given our assumptions, it must be that $\hat{\alpha} \leq \beta$. %So it suffices to show that $(\alpha,\hat{\alpha}) \notin \EXTp(k)$.

Let $\varphi$ be the function defined by $\varphi(n) = \lceil \hat{\alpha}n - 5 \log n \rceil$. Note that since $\hat{\alpha} < 1$, we have $\varphi(n) \leq n$. Let $n$ be such that
\begin{eqnarray}
f(n) \geq (\hat{\alpha}/\alpha)n + \sqrt{n}.\label{eq-assumption}
\end{eqnarray}

By Lemma~\ref{lem:negative-partial-2} (where $D_n$ is the set of strings of length~$f(n)$), there is some $x$ of length $f(n)$ such that
\begin{eqnarray}
C(x) & > & f(n) -(n-\varphi(n)+1)k - O(1) \nonumber \\
 & > & f(n) - (n - \hat{\alpha}n + 6 \log n)k - O(1)\label{eq-C-f}
\end{eqnarray}
and
\begin{equation*}
\text{for every } i \leq k, \text{~either~} \Gamma_i(x) \text{~is undefined, or~} C(\Gamma_i(x)) < \varphi(n) + 2C(\varphi(n),n)+O(1)
\end{equation*}
Since $\varphi(n) \leq n$, we have $C(\varphi(n)) \leq \log n+O(1)$, and thus $C(\varphi(n),n) \leq 2\log n+O(1)$. Thus this last equation implies:
\begin{equation}
\text{for every } i \leq k, \text{~either~} \Gamma_i(x) \text{~is undefined, or~} C(\Gamma_i(x)) < \hat{\alpha}n -  \log n + O(1)\label{eq-C-Gamma}
\end{equation}
Let us use our assumption \eqref{eq-assumption} about $f(n)$ to evaluate the right-hand side of \eqref{eq-C-f}:
\begin{eqnarray*}
f(n) - (n - \hat{\alpha}n + 6 \log n)k - O(1) & = & f(n) \left[ 1 - \frac{(1-\hat{\alpha})nk}{f(n)} \right] -  5k \log n - O(1)\\
 & \geq &  f(n) \left[ 1 - \frac{(1-\hat{\alpha})nk}{(\hat{\alpha}/\alpha)n + \sqrt{n}} \right] -  6k \log n - O(1)\\
 & \geq &  f(n) \left[ 1 - \frac{(1-\hat{\alpha})k}{(\hat{\alpha}/\alpha) + n^{-1/2}} \right] -  6k \log n - O(1)\\
 & \geq &  f(n) \left[ 1 - \frac{(1-\hat{\alpha})k}{(\hat{\alpha}/\alpha)}+ (1-\hat{\alpha})kn^{-1/2} - o(n^{-1/2}) \right] \\
 & & - 6k \log n - O(1)
\end{eqnarray*}
(the last inequality comes from the asymptotic estimate $1/(a+\varepsilon)=1/a-\varepsilon+o(\varepsilon)$). One can easily verify that
\[
1 - \frac{(1-\hat{\alpha})k}{(\hat{\alpha}/\alpha)} = \alpha
\]
so the last inequality can be rewritten as
\begin{align*}
f(n) - (n - \hat{\alpha}n + &6 \log n)k - O(1) \geq \\
&\alpha f(n) + (1-\hat{\alpha})kn^{-1/2}f(n) - o(n^{-1/2}f(n)) - 6k \log n - O(1)
\end{align*}
The term $(1-\hat{\alpha})kn^{-1/2}f(n)$ dominates the terms $o(n^{-1/2}f(n))$ and $O(1)$ and, since $f(n)=\Omega(n)$, it also dominates the term $6k \log n$. So for any $d$, if $n$ was chosen large enough, we would have
\begin{equation*}
f(n) - (n - \hat{\alpha}n + 5 \log n)k - O(1) > \alpha f(n) + d
\end{equation*}
Together with \eqref{eq-C-f} and \eqref{eq-C-Gamma}, this shows that for any $d$ we can find some $n$ and $x$ of length~$f(n)$ such that $C(x) > \alpha f(n) +d$ and $C(\Gamma_i(x)) < \hat{\alpha} n - d \leq \beta n - d$ should $\Gamma_i(x)$ be defined. This contradicts our original assumption that $(\alpha,\beta) \in \EXTp(k)$ with witness~$f$, and so finishes the proof.
\end{proof}

\begin{corollary}\label{cor:partial-not}
Let $\alpha, \beta$ be in $(0,1)$ and $k \geq 2$. If $\beta > k\alpha/(1+(k-1)\alpha)$, then $(\alpha,\beta) \notin \EXTp(k)$.
\end{corollary}

\begin{proof}
 Assume towards a contardiction that $(\alpha, \beta) \in \EXTp(k)$, witnessed by function $f$, constant $d$ and functionals $\Gamma_1, \ldots, \Gamma_k$. By Proposition~\ref{prop:sqrt-deviation} we have
\[ (\beta/\alpha) n - O(1) \leq f(n) \leq \frac{k}{1+(k-1)\alpha} n + \sqrt{n} + O(1) \]
and so $\beta / \alpha \leq \frac{k}{1+(k-1)\alpha}$. This implies that $\beta \leq k\alpha/(1+(k-1)\alpha)$, a contradiction.
\end{proof}

For $k \geq 2$, at the threshold value $\beta = k\alpha / (1+(k-1)\alpha)$, we get a positive result, but \emph{only when $\alpha$ and $\beta$ are computable}.

\begin{theorem}
Let $k \geq 1$ and $\alpha,\beta \in (0,1)$ be computable and such that $\beta = k\alpha / (1+(k-1)\alpha)$, or, equivalently, $\alpha = \beta/(k-(k-1)\beta)$. There are $k$ partial computable functionals $\Gamma_1, \ldots, \Gamma_k$ such that $\Gamma_i(\{0,1\}^{\left\lfloor{(\beta/\alpha) n}\right\rfloor}) \subseteq \{0,1\}^n$ for all $i,n$, and a constant $d$, such that when $|x|=\left\lfloor{(\beta/\alpha) n}\right\rfloor$ and $C(x) \geq \alpha|x|+d$, $C(\Gamma_i(x)) \geq \beta n - d$ for some~$i$.
\end{theorem}

\begin{proof}
We will ensure that there is a constant $e$ such that if $C(x) \geq \beta n + e$, then $C(\Gamma_i(x)) \geq \beta n - e$ for some $i$. The result will then follow by taking $d=e+\alpha$.

Note that as $\alpha$ and $\beta$ are computable, the map $n \mapsto \left\lfloor{(\beta/\alpha) n}\right\rfloor$ is also computable (though it is not computable uniformly in a code for $\beta / \alpha$). It is computable if $\beta / \alpha$ is rational, and if this is not rational, then $(\beta/\alpha) n$ is never an integer and so we can compute the floor function of $(\beta/\alpha) n$.

We prove this by induction. For $k=1$, we have $\beta=\alpha$ so it suffices to take $\Gamma_1=id$ and we are done.

Suppose the proposition holds for level~$k$, and let us prove it for $k+1$. Consider $\alpha, \beta$ with $\beta = (k+1)\alpha / (1+k\alpha)$. Let $d$ be a large constant, which we will implicitly define throughout the proof by listing the properties it must have. We let $\Gamma_{k+1}$ be the functional which on an input $x$ of size $\left\lfloor{(\beta/\alpha)n}\right\rfloor$ looks for a $\mathbb{U}$-description $p$ for~$x$ whose length belongs to~$[n + d,n+2d]$. If $d$ is large enough, such a $p$ will be found if $C(x) \leq n+d$ (here we use a classical `padding' result for Kolmogorov complexity: there exists a constant $a$ such that for any $x$, if $C(x)=k$, then for any $k' \geq k$, there exists an $\mathbb{U}$-description $q$ of $x$ whose length belongs to $[k',k'+a]$). If such a $p$ is found, $\Gamma_{k+1}(x)$ returns the prefix $p'$ of $p$ of length~$n$. Note that in that case, if $C(x) \geq \beta n+d$, then $C(p') \geq \beta n-O(1)$: indeed, to recover $p$ from $p'$ we only need $2d$ bits of information (at most), and from $p$ we can recover $x$. Thus, in this case, $\Gamma_{k+1}$ `succeeds' on $x$.

%Suppose the proposition holds for level~$k$, and let us prove it for $k+1$. Consider $\alpha, \beta$ with $\beta = (k+1)\alpha / (1+k\alpha)$. We let $\Gamma_{k+1}$ be the functional which on an input $x$ of size $\left\lfloor{(\beta/\alpha)n}\right\rfloor$ looks for a $\mathbb{U}$-description $p$ for~$x$ of length~$ \leq n$, and outputs the first such~$p$ it finds (if necessary, padded with $O(1)$ zeroes at the end to give it exactly length~$n$). If $C(x) \geq \alpha|x|+d$ (i.e., $C(x) \geq \beta n +d)$, then $p=\Gamma_{k+1}(x)$ has the same Kolmogorov complexity up to $O(1)$ as $x$ does, since $x$ can be recovered from~$p$. That is, $C(p) \geq \beta n + d = \beta |p|+d+O(1)$, thus in this case $\Gamma_{k+1}$ `succeeds' on mapping an $\alpha$-random~$x$ to a $\beta$-random~$p$.

However, $\Gamma_{k+1}(x)$ could be undefined, which as we saw would mean that $C(x)>n +d$. In this case, we use the induction hypothesis at level~$k$: Setting $\gamma=\beta/(k-(k-1)\beta)$, there are $k$ partial functionals $\Phi_1,  \ldots, \Phi_k$ such that $\Phi_i(\{0,1\}^{\left\lfloor{(\beta/\gamma) n}\right\rfloor}) \subseteq \{0,1\}^n$ for all $i,n$, and a constant $e$ such that when $|y|=\left\lfloor{(\beta/\gamma) n}\right\rfloor$ and $C(y) \geq \beta n +e$, then $C(\Phi_i(y)) \geq \beta n - e$ for some~$i$.

For $i \leq k$, let $\Gamma_i$ be the functional which does the following. On an input $x$ of length~$\lfloor(\beta/\alpha)n\rfloor$, it computes the prefix $x^-$ of $x$ of length $\lfloor(\beta/\gamma)n\rfloor$, and returns $\Phi_i(x^-)$. We claim than when $\Gamma_{k+1}(x)$ is undefined, i.e., when $C(x)\geq n +d$, one of the $\Gamma_i$, $i \leq k$ must succeed.

Indeed, when $C(x) \geq n +d $, because $x^-$ is obtained from $x$ by removing only $\left\lfloor{(\beta/\alpha) n}\right\rfloor - \left\lfloor{(\beta/\gamma) n}\right\rfloor$ bits (which is computable knowing $n$), we must have
\begin{eqnarray*}
C(x^-) & \geq & n+d - \big( \left\lfloor{(\beta/\alpha) n}\right\rfloor - \left\lfloor{(\beta/\gamma) n}\right\rfloor \big) - O(1)\\
 & \geq & n + \frac{\beta}{\gamma}n - \frac{\beta}{\alpha}n - O(1)\\
 & \geq &  n \big(1 + (k - (k-1) \beta) - ((k+1) - k\beta )\big) +d - O(1)\\
 & \geq & \beta n+d - O(1)\\
\end{eqnarray*}
where the $O(1)$ term is independent of all other terms (it only depends on the choice of universal machine $\mathbb{U}$). Thus, if $d$ is chosen large enough, we have $C(x^-) \geq \beta n +e$ and the induction hypothesis  proves that in this case, one of the $\Phi_i(x^-)$ returns a string $y$ of length~$n$ with $C(y) \geq \beta n - O(1)$.
\end{proof}

\begin{theorem}\label{thm:threshold-not-computable}
Let $k \geq 2$. If $\alpha$, $\beta \in (0,1)$ are such that $\beta = k\alpha / (1+(k-1)\alpha)$, but are not computable (note that the relation between $\alpha$ and $\beta$ implies that they are either both computable or both incomputable), then $(\alpha, \beta) \notin \EXTp(k)$ and a fortiori, $(\alpha, \beta) \notin \EXT(k)$.
\end{theorem}

\begin{proof}
Suppose for the sake of contradiction that $\Gamma_1, \ldots, \Gamma_k$ and $f$ witness that $(\alpha, \beta) \in \EXTp(k)$. By Proposition~\ref{prop:sqrt-deviation} we have
\[ (\beta/\alpha) n - O(1) \leq f(n) \leq \frac{k}{1+(k-1)\alpha} n + \sqrt{n} + O(1). \]
But $(\beta / \alpha) = \frac{k}{1+(k-1)\alpha}$, and so the computable function $f(n)/n$ would converge to $\beta/\alpha$ at computable speed (namely $|f(n)/n - (\beta/\alpha)|<n^{-1/2}$), making $\beta/\alpha$ computable. But $\beta/\alpha= k / (1+(k-1)\alpha)$, so this would make $\alpha$ computable, a contradiction. (This is where we use that $k \geq 2$).
\end{proof}

\section{Going beyond constant-size advice}

The tight inequality $\beta < k\alpha/(1+(k-1)\alpha)$ we have obtained allows us to get a more precise version of Theorem~\ref{thm:vv}:

\begin{theorem}\label{thm:vv-improved}
Fix $0<\alpha<1$ and suppose there is a partial computable function $E(.,.)$, a linear function $f$, and a constant $m$, with the property that for every $n$, for every $\sigma$ of length $f(n)$ such that $C(\sigma) \geq \alpha |\sigma|$, there exists a string $a_\sigma$ of length~$h$ such that $\tau=E(\sigma,a_\sigma)$ has length~$n$ and $C(\tau) \geq \beta |\tau|$. Then  $\beta \leq 1 - \left( \frac{1-\alpha}{\alpha}\right)2^{-h}+o(2^{-h})$. Moreover, this bound is tight.
\end{theorem}

\begin{proof}
As we discussed in the introduction, having~$h$ bits of advice is equivalent to having $k=2^h$ functionals. The result then follows from the tight bound $\beta \leq k\alpha/(1+(k-1)\alpha) = 2^h /(1+(2^h-1)\alpha)$ that arises from the results in the last two sections, and the straightforward asymptotic estimate $2^h /(1+(2^h-1)\alpha) = (1-\alpha)/\alpha \cdot 2^{-h}+o(2^{-h})$.
\end{proof}

Zimand also studied the case where the amount of advice $h$ is no longer constant but is a (computable) function of~$n$. He showed the following theorem (which we slightly reformulate to fit our framework), essentially showing that if we allow any unbounded amount of advice, then we can asymptotically achieve dimension~$1$:

\begin{theorem}[Zimand~\cite{Zimand2011}]
Let $f, h$ be computable functions such that $f(n) \geq n$ and $\log(f(n)/n) = o(h(n))$. Then there exist a computable function~$E(.,.)$ and a constant~$d$ such that for every~$n$, if $|x|=f(n)$ and $|a|=h(n)$, then $|E(x,a)|=n$ and if moreover $x$ has length~$f(n)$ and $C(x) \geq n+d$, then for some $a$ of length~$h(n)$, $C(E(x,a)) \geq n-\frac{f(n)}{2^{h(n)/2}} \geq n-o(n)$.
\end{theorem}

This says for example, when $f(n)=2n$, and $h(n) \rightarrow \infty$, that using $h(n)$ bits of advice one can turn a string of length~$2n$ and dimension $1/2$ into a string of dimension $1-\varepsilon(n)$, where $\varepsilon(n) \rightarrow 0$.

By using a variant of our random graph argument from Section~\ref{sec:total-two}, we can get a slight improvement of this result, namely, we can prove the following.

\begin{theorem}
Let $f, h$ be computable functions such that $f(n) \geq n$. Then there exist a computable function~$E(.,.)$ and a constant~$d$ such that for every~$n$, if $|x|=f(n)$ and $|a|=h(n)$, then $|E(x,a)|=n$ and if moreover $C(x) \geq n+d$, then for some $a$ of length~$h(n)$, $C(E(x,a)) \geq n-\frac{f(n)-n}{2^{h(n)}}-d$.
\end{theorem}

\noindent (Note in particular that we no longer need to assume $\log(f(n)/n) = o(h(n))$).

\begin{proof}
Let $k(n)=2^{h(n)}$ and set $\psi(n)= n- \frac{f(n)-n}{k(n)}$. Let us again reformulate the problem into a combinatorial one.

%First, remark once again that we can view the $E$ we are trying to build as a computable function which proposes a list of strings; more precisely, one can consider instead of~$E$ the computable function $\tilde{E}$ such that if $|x|=f(n)$, $\tilde{E}(x) = \{E(x,a) \mid |a|=h(n)\}$. Moreover, we can assume the elements of $\tilde{E}(x)$ to be pairwise distinct (as we did in the proof of Theorem~\ref{thm:graph-correspondence}).

The existence of such an $E$ will follow from the following fact which natural analogue of Theorem~\ref{thm:graph-correspondence}: there exist $d$ and a computable sequence $(G_n)$, where for all $n$, $G_n$ is a $k(n)$-hypergraph with $2^n$ vertices and $2^{f(n)}$ edges, such that for every set~$U$ of vertices of $G_n$ of size $<2^{\psi(n)}$, $e(U) \leq 2^{n+d}$. However, in order to simplify our calculations, for this proof only we shall define a $k$-hyperedge over a set of vertices $V$ to be a $k$-tuple of elements of $V$\footnote{The reason we defined hyperedges to be unordered and without repetitions up to this point is that Lemma~\ref{lem:seq} seems easier to prove in this setting.}. For a subset $U$ of $V$, $e(U)$ is the number of hyperedges all of whose coordinates belong to $U$.

Let us briefly check that this fact implies our theorem. The idea is almost the same as for Theorem~\ref{thm:graph-correspondence}. Suppose there exists such a sequence of $k(n)$-hypergraphs $(G_n)$. One computably labels the vertices of $G_n$ with strings of length~$n$ and the edges with strings of length $f(n)$. Define $E(x,a)$ to be the $a$-th coordinate of the edge labeled $x$, where $a$ is a string of length $h(n)$, identified with an integer in $[1,2^{h(n)}]=[1,k(n)]$.

Now take $U_n$ to be the set of (labels of) vertices $y$ of $G_n$ such that $C(y) < \psi(n)$ (which implies $|U_n|<2^{\psi(n)}$); if indeed $e(U_n) < 2^{n+d}$ we can ensure as before that all edges in $E(U_n)$ have complexity $<n+d+O(1)$. The contrapositive says that if $C(x)> n+d+O(1)$, then one of the coordinates of the edge (labeled by) $x$ has complexity $>\psi(n)$, that is, some $E(x,a)$ has complexity $>\psi(n)$.

%(The proof that this implies our theorem follows the same : one labels the vertices of $G_n$ with strings of length~$n$, and identifies each $\tilde{E}(x)$ with $|x|=f(n)$ to an the hyperedge incident to the members of $\tilde{E}(x)$, and considers the set $U_n=\{y \in G_n \mid C(y) < \psi(n)\}$; if indeed $e(U) < 2^{n+d}$ we can ensure as before that all edges in $E(U)$ have complexity $<n+d+O(1)$).

So now it remains to prove the combinatorial fact. For each~$n$, consider the random $k(n)$-hypergraph $G_n$ with $2^{n}$ vertices and where each $k(n)$-hyperedge has probability $2^{f(n)-nk(n)+3}$ to be put in the hypergraph, with~$D$ a large constant. Now there are $(2^n)^{k(n)}=2^{n k(n)}$ potential edges, so the expectation of the number of edges in $G_n$ is $2^{f(n)+3}$. By the Chernoff bound $G_n$ has at least $2^{f(n)}$ hyperedges with probability $>1/2$. For any fixed set~$U$ of vertices of size $<2^{\psi(n)}$, there are at most $2^{\psi(n)  k(n)}$ hyperedges all of whose coordinates are in~$U$, thus we have
\[
\mathbb{E}(e(U))=2^{\psi(n) k(n)} \cdot 2^{f(n)-nk(n)+3} \leq {2^{n+3}}
\]
(for the last inequality we use the definition of $\psi(n)$). Thus, by the Chernoff bound, the probability that $e(U)>2^{n+6}$ is less than $2^{-2^{n+6}}$. Thus, the probability that \emph{some}~$U$ of size $<2^{\psi(n)}$ has $e(U)>2^{n+6}$ is bounded by
\[
{2^n \choose \lfloor 2^{\psi(n)} \rfloor} \cdot 2^{-2^{n+6}}
\]
Using the fact that ${a \choose b} = o(2^a)$ (this is because ${a \choose b}$ is maximized for $b=\lfloor a/2 \rfloor$, and Stirling's formula implies that ${a \choose \lfloor a/2 \rfloor} \sim \frac{2^a}{\sqrt{\pi a/2}}=o(2^a)$), we see that the above expression tends to $0$ as $n$ tends to infinity. In particular, for $n$ large enough, this probability is smaller than $1/2$, thus there exists a graph $G_n$ as wanted.
\end{proof}

\bibliography{References}
\bibliographystyle{alpha}

\end{document}